\documentclass{amsart}

\usepackage{centernot}
\usepackage{dsfont}
\usepackage{amsmath}
\usepackage{leftidx}
\usepackage{amsthm}
\usepackage{amsfonts}
\usepackage{mathtools}
\usepackage{color}
\usepackage[margin=40mm]{geometry}

\newcommand{\reals}{\mathbb{R}}

\newcommand{\complex}{\mathbb{C}}

\newcommand{\angles}[1]{\left\langle #1 \right\rangle}

\newcommand{\paraa}[1]{\big(#1\big)}
\newcommand{\parab}[1]{\Big(#1\Big)}

\newcommand{\spacearound}[1]{\quad#1\quad}
\newcommand{\equivalent}{\spacearound{\Leftrightarrow}}

\newcommand{\qtext}[1]{\quad\text{#1}\quad}
\newcommand{\qqtext}[1]{\qquad\text{#1}\qquad}
\newcommand{\qand}{\qtext{and}}
\newcommand{\qqand}{\qqtext{and}}

\newtheorem{theorem}{Theorem}[section]

\newtheorem{lemma}[theorem]{Lemma}
\newtheorem{proposition}[theorem]{Proposition}
\newtheorem{example}[theorem]{Example}
\theoremstyle{definition}
\newtheorem{definition}[theorem]{Definition}
\theoremstyle{remark}
\newtheorem{remark}[theorem]{Remark}
\numberwithin{equation}{section}

\renewcommand{\d}{\partial}

\renewcommand{\mid}{\mathds{1}}
\newcommand{\Xp}{X_{+}}
\newcommand{\Xm}{X_{-}}
\newcommand{\Xz}{X_{z}}
\newcommand{\omegap}{\omega_{+}}
\newcommand{\omegam}{\omega_{-}}
\newcommand{\omegaz}{\omega_{z}}
\newcommand{\Sthreeq}{S^{3}_{q}}

\newcommand{\Uqsu}{\mathcal{U}_q(\textrm{su}(2))}

\newcommand{\A}{\mathcal{A}}

\newcommand{\U}{\mathcal{U}}
\newcommand{\lt}{\triangleright}
\newcommand{\rt}{\triangleleft}
\newcommand{\one}[1]{#1_{(1)}}
\newcommand{\two}[1]{#1_{(2)}}

\newcommand{\hh}{\hat{h}}
\newcommand{\Gammat}{\tilde{\Gamma}}
\newcommand{\sigmatau}{(\sigma,\tau)}

\newcommand{\sigmah}{\hat{\sigma}}
\newcommand{\sigmahz}{\sigmah^0}
\newcommand{\tauh}{\hat{\tau}}
\newcommand{\tauhz}{\tauh^0}
\newcommand{\sigmat}{\tilde{\sigma}}
\newcommand{\taut}{\tilde{\tau}}
\newcommand{\TSigma}{T\Sigma}
\newcommand{\AstX}{\{\A,\{X_a\}_{a\in I}\}}
\newcommand{\st}{(\sigma,\tau)}

\newcommand{\eh}{\hat{e}}
\newcommand{\nablat}{\tilde{\nabla}}
\newcommand{\Kh}{\hat{K}}
\newcommand{\id}{\operatorname{id}}

\newcommand{\comR}[1]{[#1]_R}
\newcommand{\otimesC}{\otimes_{\complex}}
\newcommand{\End}{\operatorname{End}}
\newcommand{\Mat}{\operatorname{Mat}}
\newcommand{\MatNC}{\Mat_{N}(\complex)}

\newcommand{\Curv}{\operatorname{Curv}}
\newcommand{\idA}{\id_{\A}}
\newcommand{\Mp}{M_p}
\newcommand{\Xt}{\tilde{X}}
\newcommand{\Sigmas}{\Sigma^{\ast}}

\title{On the geometry of $(\sigma, \tau)$-algebras}
\author{Joakim Arnlind and Kwalombota Ilwale}

\address[Joakim Arnlind]{Dept. of Math.\\
Link\"oping University\\
581 83 Link\"oping\\
Sweden}
\email{joakim.arnlind@liu.se}

\address[Kwalombota Ilwale]{Dept. of Math.\\
Link\"oping University\\
581 83 Link\"oping\\
Sweden}
\email{kwalombota.ilwale@liu.se}

\subjclass[2010]{46L87}
\keywords{}

\begin{document}

\maketitle

\begin{abstract}
  We introduce $\st$-algebras as a framework for
  twisted differential calculi over noncommutative, as well as
  commutative, algebras with motivations from the theory of
  $\sigma$-derivations and quantum groups. A $\st$-algebra consists
  of an associative algebra together with a set of $\st$-derivations,
  and corresponding notions of $\st$-modules and connections are
  introduced.  We prove that $\st$-connections exist on projective
  modules, and introduce notions of both torsion and curvature, as
  well as compatibility with a hermitian form, leading to the
  definition of a Levi-Civita $\st$-connection. To illustrate the
  novel concepts, we consider $\st$-algebras and connections over
  matrix algebras in detail.
\end{abstract}

\section{Introduction}

\noindent
In the context of mathematics and mathematical physics, the process of
deforming well-known objects into new ones has proven to be a rich
source of new phenomena. Important examples stem from the theory of
quantum groups and deformation quantization, which turn out to be well
suited to study both quantum and noncommutative geometry. The
(noncommutative) differential geometry of such deformations has been
studied by many authors, and there is a fair amount of results related
to both differential and Riemannian structures as well as aspects of
symmetry (see
e.g. \cite{m:quantum.groups.ncg,cff:gravityncgeometry,m:quantum.braided.group,cl:isospectral,bm:starCompatibleConnections,bm:quantum.riemannian.geometry}). For
the most part, these approaches are based on finding noncommutative
analogues of differential forms in a framework which is independent of
derivations (or vector fields). However, more derivation based
approaches to noncommutative differential geometry have also been
considered in parallel (see
e.g. \cite{dv:calculDifferentiel,dbm:central.bimodules,ac:ncgravitysolutions,r:leviCivita,aw:cgb.sphere,aw:curvature.three.sphere,a:lc.braided}). In
this context one should also mention the rich literature related to
the spectral approach to noncommutative geometry pioneered by
A. Connes \cite{c:ncgbook}, even though analytic aspects are not the
main focus of our work.

From a slightly different point of view, deformations of algebraic
structures, motivated by e.g. discretizations of vector fields, has
been studied in terms of $\sigmatau$-derivations and Hom-Lie algebras
(see e.g. \cite{hls:sigmaderivation,ls:quasi-deformations}). In
contrast to many of the above approaches, one considers deformations
of derivations into $\sigmatau$-derivations. A $\sigmatau$-derivation
on an associative algebra $\A$ is a $\complex$-linear map $X:\A\to\A$ such that
there exist endomorphisms $\sigma,\tau:\A\to\A$ fulfilling
\begin{align*}
  X(fg) = \sigma(f)X(g)+X(f)\tau(g),
\end{align*}
for all $f,g\in\A$. Such twisted derivations are not only relevant to
noncommutative algebras; for instance, one may consider the Jackson
$q$-derivative acting on functions on the real line as
\begin{align*}
  (\d_qf)(t) = \frac{f(qt)-f(t)}{t(q-1)}
\end{align*}
for some $q>1$, which is a $\sigmatau$-derivation with
$\sigma(f)(t)=f(qt)$ and $\tau(f)(t)=f(t)$. This ``discrete
derivative'' is a crucial ingredient when constructing $q$-deformed
Witt algebras over the Laurent polynomials $\complex[t,t^{-1}]$
\cite{hls:sigmaderivation}. However, differential geometry based on
$\sigmatau$-derivations has not been considered to the same
extent. Motivated by the case study of the quantum 3-sphere, described in
\cite{ail:qdeformed,ail:lc.spheres}, we introduce a framework for studying modules
and connections over algebras together with a set of
$\sigmatau$-derivations. In this setting, we require that a (left)
$\A$-module $M$ comes equipped with $\complex$-linear maps $\sigmah,\tauh:M\to M$ such
that
\begin{align*}
  \sigmah(fm) = \sigma(f)\sigmah(m)\qquad
  \tauh(fm) = \tau(f)\tauh(m)
\end{align*}
for $f\in\A$ and $m\in M$, and for a $\sigmatau$-derivation $X$, a $\sigmatau$-connection
$\nabla_X:M\to M$ is required to satisfy a twisted Leibniz' rule
\begin{align*}
  \nabla_X(fm) = \sigma(f)\nabla_Xm + X(f)\tauh(m).
\end{align*}
in correspondence with $X(fg)=\sigma(f)X(g)+X(f)\tau(g)$. Furthermore,
we introduce metric compatibility and torsion of
$\sigmatau$-connections, as well as curvature, and prove that
$\sigmatau$-connections exist on projective modules, and introduce a
concept of Levi-Civita connections. Along the way, we present a number of
examples to illustrate the novel concepts. In particular,
$\st$-algebras over matrix algebras are considered in Section~\ref{sec:lc.matrix}.

\section{$(\sigma,\tau)$-algebras and $\Sigma$-modules}

\noindent
In a geometric interpretation of derivation based differential calculi
over an algebra $\A$, one chooses a Lie algebra of derivations as a
substitute for vector fields, in analogy with the fact that the set of
vector fields on a smooth manifold can be identified with the set of
derivations of the algebra of functions on the manifold. A derivation
of a (possibly noncommutative) algebra $\A$ is a $\complex$-linear map
$X:\A\to\A$ satisfying the so called Leibniz' rule
\begin{align}\label{eq:Leibniz.rule}
  X(fg) = fX(g) + X(f)g
\end{align}
for all $f,g\in\A$. In several different contexts, ``twisted
derivations'' appear in a natural way, satisfying a deformed version
of \eqref{eq:Leibniz.rule}
\begin{align}\label{eq:sigmatau.der.rule}
  X(fg) = \sigma(f)X(g) + X(f)\tau(g)
\end{align}
where $\sigma,\tau:\A\to\A$ are algebra endomorphisms. Such maps are
called \emph{$\sigmatau$-derivations}
(cf. Definition~\ref{def:sigmatau.derivation}). For instance,
different types of discretized derivatives will typically satisfy
\eqref{eq:sigmatau.der.rule} rather than the standard Leibniz'
rule. Hence, $\sigmatau$-derivations are not only relevant in the case
of noncommutative or ``quantum'' geometry, but also for commutative
algebras. Now, we are interested in the following question: Can one
construct an algebraic framework to study geometry based on
$\sigmatau$-derivations?

In this section we introduce $\st$-algebras as associative algebras
together with a set of $\sigmatau$-derivations, and define
$\Sigma$-modules as modules over $\sigmatau$-algebras, that carry an
extension of $\sigma$ and $\tau$ as module maps. Although already
discussed to some extent, let us start by properly defining
$\sigmatau$-derivations.

\begin{definition}\label{def:sigmatau.derivation}
  Let $\A$ be an associative algebra over $\complex$ and let $\sigma$
  and $\tau$ be endomorphisms of $\A$. A $\complex$-linear map
  $X : \A \to \A$ is called a $(\sigma, \tau)$-derivation of $\A$ if
  \begin{eqnarray}\label{(sigma, tau der}
    X(fg) = \sigma(f)X(g) + X(f)\tau(g)
  \end{eqnarray}
  for every $f,g \in \A.$
\end{definition}

\noindent A canonical example of a $\st$-derivation on any associative
algebra $\A$ is given by
\begin{align}\label{eq:sigma.tau.inner.derivation}
  X(f) = \tau(f)-\sigma(f),
\end{align}
for arbitrary endomorphisms $\sigma,\tau:\A\to\A$. For instance, letting $\A$ denote
the (commutative) algebra of functions on the real line, one notes
that the difference operator
\begin{align*}
  \d(f)(x) = f(x+\hbar)-f(x)
\end{align*}
for $\hbar\in\reals$, is a $(\sigma,\tau)$-derivation of the form as in
\eqref{eq:sigma.tau.inner.derivation}, with $\tau(f)(x)=f(x+\hbar)$
and $\sigma(f)(x)=f(x)$. Note that $\sigmatau$-derivations of the form
as in \eqref{eq:sigma.tau.inner.derivation} are simultaneously
$\sigmatau$- and $(\tau,\sigma)$-derivations; i.e.
\begin{align*}
  X(fg) = \sigma(f)X(g) + X(f)\tau(g) = \tau(f)X(g) + X(f)\sigma(g).
\end{align*}
In the next definition, we introduce the main object of our study; an
associative algebra together with a set of $\sigmatau$-derivations.

\begin{definition}
  A \emph{$\sigmatau$-algebra $\Sigma = (\A,\{X_a\}_{a\in I})$} is a
  pair where $\A$ is an associative unital algebra over $\complex$, $I$ is a
  set, and $X_a$ is a $(\sigma_a,\tau_a)$-derivation of $\A$ for
  $a\in I$.
\end{definition}

\noindent
As we are interested in linear combinations of elements of the set
$\{X_a\}_{a\in I}$, we introduce the vector space generated by this
set. Note, however, that the sum of two $\sigmatau$-derivations (for
\emph{different} $\sigma$'s and $\tau$'s) is not necessarily a
$\sigmatau$-derivation. Therefore, a general element of the vector
space is only a $\complex$-linear map.

\begin{definition}
  Let $\Sigma=\AstX$ be a $\sigmatau$-algebra and define
  $\TSigma\subseteq\End_{\complex}(\A)$ to be the complex vector space
  generated by $\{X_a\}_{a\in I}$. We call $\TSigma$ the \emph{tangent
    space of $\Sigma$}.
\end{definition}

\noindent 
Morphisms of $\sigmatau$-algebras are algebra morphisms preserving the
action of the $\sigmatau$-derivations as follows.

\begin{definition}
  Let $\Sigma=(\A,\{X_a\}_{a\in I})$ and
  $\Sigma'=(\A',\{Y_\alpha\}_{\alpha\in J})$ be $\st$-algebras. A
  \emph{$\st$-algebra homomorphism $(\phi,\psi):\Sigma\to\Sigma'$} is
  a pair of maps such that $\phi:\A\to\A'$ is an algebra homomorphism
  and $\psi:T\Sigma'\to T\Sigma$ is a $\complex$-linear map such that
  \begin{align}\label{eq:stmorphism.def}
    \phi\paraa{\psi(Y)(f)} = Y\paraa{\phi(f)} 
  \end{align}
  for all $Y\in\TSigma'$ and $f\in\A$. Furthermore, $(\phi,\psi)$ is
  called a \emph{$\st$-algebra isomorphism} if $\phi$ is an algebra
  isomorphism and $\psi$ is a bijection.
\end{definition}

\noindent
Note that if $\phi$ is invertible, then \eqref{eq:stmorphism.def} determines $\psi$ uniquely as
$\psi(Y) = \phi^{-1}\circ Y\circ\phi$. Thus, if $\phi$ is an algebra isomorphism then there
exists at most one $\psi$ such that $(\phi,\psi)$ is a $\st$-algebra
morphism.

Now, let us proceed to introduce modules over $\sigmatau$-algebras.

\begin{definition}
  Let $\Sigma=(\A,\{X_a\}_{a\in I})$ be a $\sigmatau$-algebra. A left
  $\Sigma$-module is a pair $(M,\{(\sigmah_a,\tauh_a)\}_{a\in I})$
  where $M$ is a left $\A$-module and $\sigmah_a,\tauh_a:M\to M$ are
  $\complex$-linear maps such that
  \begin{align*}
    &\sigmah_a(fm) = \sigma_a(f)\sigmah_a(m)\\
    &\tauh_a(fm) = \tau_a(f)\tauh_a(m)
  \end{align*}
  for $f\in\A$, $m\in M$ and $a\in I$.
\end{definition}

\noindent
We say that the $\Sigma$-module $(M,\{(\sigmah_a,\tauh_a)\}_{a\in I})$
is \emph{free} respectively \emph{projective}, if $M$ is a free
respectively projective $\A$-module.  For completeness, let us also
state the definition of a right $\Sigma$-module, as well as a
$\Sigma$-bimodule.

\begin{definition}
  Let $\Sigma=(\A,\{X_a\}_{a\in I})$ be a $\sigmatau$-algebra. A right
  $\Sigma$-module is a pair $(M,\{(\sigmah_a,\tauh_a)\}_{a\in I})$
  where $M$ is a right $\A$-module and $\sigmah_a,\tauh_a:M\to M$ are
  $\complex$-linear maps such that
  \begin{align*}
    &\sigmah_a(mf) = \sigmah_a(m)\sigma_a(f)\\
    &\tauh_a(mf) = \tauh_a(m)\tau_a(f)
  \end{align*}
  for $f\in\A$, $m\in M$ and $a\in I$.
\end{definition}

\begin{definition}
  Let $\Sigma=(\A,\{X_a\}_{a\in I})$ be a $\sigmatau$-algebra and let
  $M$ be a $\A$-bimodule. A $\Sigma$-bimodule $(M,\{\sigmah_a,\tauh_a\}_{a\in I})$ is
  a left $\Sigma$-module which is also a right $\Sigma$-module, i.e.
  \begin{align*}
    &\sigmah_a(fmg) = \sigma_a(f)\sigmah(m)\sigma_a(g)\\
    &\tauh_a(fmg) = \tau_a(f)\tauh_a(m)\tau_a(g)
  \end{align*}
  for $f,g\in\A$, $m\in M$ and $a\in I$.
  
\end{definition}

\noindent
Morphisms of $\Sigma$-modules are readily introduced as module
morphisms intertwining the corresponding $\sigmah_a$ and $\tauh_a$.
\begin{definition}
  Let $\Sigma=(\A,\{X_a\}_{a\in I})$ be a $\sigmatau$-algebra and let
  $(M_1,\{(\sigmah_a,\tauh_a)\}_{a\in I})$ and
  $(M_2,\{(\sigmat_a,\taut_a)\}_{a\in I})$ be left (right)
  $\Sigma$-modules. A left (right) \emph{$\Sigma$-module homomorphism}
  is a left (right) $\A$-module homomorphism $\phi:M_1\to M_2$ such
  that
  \begin{align*}
    \phi\paraa{\sigmah_a(m)} = \sigmat_a\paraa{\phi(m)}\qand
    \phi\paraa{\tauh_a(m)} = \taut_a\paraa{\phi(m)}
  \end{align*}
  for $m\in M_1$ and $a\in I$. Moreover, if $\phi$ is a left (right) module
  isomorphism we say that $\phi$ is a left (right) $\Sigma$-module isomorphism.
\end{definition}

\noindent
Let us now illustrate the above concepts with a simple example,
showing that every free module can be given the structure of a
$\Sigma$-module.

\begin{example}\label{ex:free.module.sigma.structure}
  Let $\Sigma=(\A,\{X_a\}_{a\in I})$ be a $\sigmatau$-algebra and let
  $\A^{n}$ be a free left $\A$-module with basis
  $e_{1},\ldots, e_{n}$. Defining
  \begin{eqnarray*}
    \hat{\sigma}^0_a(m) = \sigma_a(m^i)e_i\qand
    \hat{\tau}^0_a(m) = \tau_a(m^i)e_i
  \end{eqnarray*}
  for $m=m^ie_i\in \A^n$, it follows immediately that
  $(\A^n,\{(\hat{\sigma}^0_a,\hat{\tau}^0_a)\}_{a \in I})$ is indeed a
  left $\Sigma$-module. The construction can be slightly generalized
  by noting that one may also set
  \begin{align*}
    \sigmah_a(m) = \sigma_a(m^i)\sigmah(e_i)\qand
    \tauh_a(m) = \tau_a(m^i)\tauh(e_i)  
  \end{align*}
  for an arbitrary choice of $\sigmah(e_i),\tauh(e_i)\in\A^n$. Note that here,
  and in the following, we use the convention that repeated
  indices are summed over unless otherwise stated.
\end{example}

\noindent
The following proposition states that if $M$ has a $\Sigma$-module
structure then for any module endomorphism $T : M \to M,$ the
$\A$-module $T(M)$ can be endowed with a compatible $\Sigma$-module
structure.

\begin{proposition}\label{prop:TM}
  Let $(M,\{(\hat{\sigma}_{a}, \hat{\tau}_{a})\}_{a \in I})$ be a left (right)
  $\Sigma$-module and let $T : M \to M$ be a left (right) $\A$-module
  endomorphism. Then
  $(T(M),\{(\tilde{\sigma}_{a}, \tilde{\tau}_{a})\}_{a \in I})$ is a a
  left (right) $\Sigma$-module where
  $\tilde{\sigma}_{a} = T \circ \hat{\sigma}_{a}|_{T(M)}$ and
  $\tilde{\tau}_{a} = T \circ \hat{\tau}_{a}|_{T(M)}.$
\end{proposition}

\begin{proof}
  Given that $(M, \{(\hat{\sigma}_{a}, \hat{\tau}_{a})\}_{a \in I})$
  is a $\Sigma$-module, let $m \in T(M)$ and $f \in \A,$ one has
  \begin{align*}
    \tilde{\sigma}_{a}(fm)
    &= (T \circ \hat{\sigma}_{a})(f m)
      = T (\sigma_{a}(f)\hat{\sigma}_{a}(m))
      = \sigma_{a}(f)T(\hat{\sigma}_{a}(m))
      = \sigma_{a}(f)\tilde{\sigma}_{a}(m)\\
    \tilde{\tau}_{a}(fm)
    &= T \circ \hat{\tau}_{a}(fm)
      = T(\tau_{a}(f)\hat{\tau}_{a}(m))
      = \tau_{a}(f)T(\hat{\tau}_{a}(m))
      =\tau_{a}(f)\tilde{\tau}_{a}(m).   
  \end{align*}
  showing that
  $(T(M), \{(\tilde{\sigma}_{a}, \tilde{\tau}_{a})\}_{a \in I})$ is a
  left $\Sigma$-module. The proof for right $\Sigma$-modules is analogous.
\end{proof}

\noindent
Applying Proposition~\ref{prop:TM} to the free $\Sigma$-module
$(\A^n,\{\sigmahz_a,\tauhz_a\}_{a\in I})$ of
Example~\ref{ex:free.module.sigma.structure}, and choosing $T=p$ to be
a projection, one obtains a projective $\Sigma$-module
\begin{align*}
  (p(\A^n),\{p\circ\sigmahz_a,p\circ\tauhz_a\}_{a\in I}).
\end{align*}
Conversely, every projective $\Sigma$-module can be realized in this
way, as the following result shows.

\begin{proposition}\label{proj mod iso}
  Let $(M, \{(\tilde{\sigma}_{a}, \tilde{\tau}_{a})\}_{a \in I})$ be a
  $\Sigma$-module such that $M$ is a finitely generated projective
  $\A$-module. Then there exists a free $\Sigma$-module
  $(\A^{n}, \{(\hat{\sigma}_{a}, \hat{\tau}_{a})\}_{a \in I})$ and a
  projection $p : \A^{n} \to \A^{n}$ such that
  \begin{align*}
    (p(\A^{n}), \{(p \circ \hat{\sigma}_{a}, p \circ \hat{\tau}_{a})\}_{a \in I})
    \simeq (M, \{(\tilde{\sigma}_{a}, \tilde{\tau}_{a})\}_{a \in I})
  \end{align*}
  and $[\hat{\sigma}_{a}, p] = [\hat{\tau}_{a}, p] = 0$ for $a \in I.$
\end{proposition}
\begin{proof}
  Assume that $M$ is finitely generated with generators
  $e_{1}, \ldots, e_{n}$, and let $\eh_1,\ldots,\eh_n$ be a basis of
  $\A^n$. Defining $\phi : \A^{n} \to M$ by
  $\phi(m^{i}\hat{e}_{i}) = m^{i}e_{i}$ one notes that $\phi$ is
  surjective.  Since $M$ is a projective module, there exists a
  homomorphism $ \psi : M \to \A^{n}$ such that
  $\phi \circ \psi = \id_{M}.$ Defining $p : \A^{n} \to \A^{n}$
  as $p = \psi \circ \phi$ one finds that
  \begin{align*}
    p^{2} = \psi \circ \phi \circ \psi \circ \phi = \psi \circ \phi = p,
  \end{align*}
  since $\phi \circ \psi = \id_{M}$.  Now, let
  $\hat{\phi} = \phi|_{p(\A^{n})} : p(\A^{n}) \to M$ be the
  restriction of $\phi$ to $p(\A^{n})$. Since
  \begin{align*}
    p(\psi(m)) = \psi \circ \phi \circ \psi(m) = \psi(m),
  \end{align*}
  one concludes that $\psi(m) \in p(\A^{n})$, showing that
  $\hat{\phi}$ is surjective. Moreover, $\hat{\phi}$ is injective
  since if $m\in p(\A^n)$ and $\phi(m)=0$, then
  \begin{align*}
    m = p(m) = \psi\paraa{\phi(m)} = 0,
  \end{align*}
  and we conclude that $\hat{\phi}:p(\A^n)\to M$ is a module isomorphism.
  
  Defining $\sigmah_a,\tauh_a:\A^n\to\A^n$ as
  \begin{align*}
    \hat{\sigma}_{a} = \psi \circ \tilde{\sigma}_{a} \circ \phi\qand
    \hat{\tau}_{a} = \psi \circ \tilde{\tau}_{a} \circ \phi,
  \end{align*}
  it follows that
  \begin{align*}
    \hat{\sigma}_{a}(fm) &= \psi\paraa{\tilde{\sigma}_{a} (\phi(fm))}
      = \psi\paraa{ \tilde{\sigma}_{a}(f\phi(m))}
      = \psi\paraa{\sigma_{a}(f)\tilde{\sigma}_{a}\phi(m))}\\
    &= \sigma_{a}(f) \psi \paraa{ \tilde{\sigma}_{a} (\phi(m))}
      = \sigma_{a}(f)\hat{\sigma}_{a}(m),
  \end{align*}
  as well as
  \begin{align*}
    \hat{\tau}_{a}(fm)
    &= \psi\paraa{\tilde{\tau}_{a}(\phi(fm))}
      =\psi\paraa{\tilde{\tau}_{a}(f\phi(m))}
      = \psi\paraa{\tau_{a}(f)\tilde{\tau}_{a}(\phi(m))}\\
    &= \tau_{a}(f)\psi\paraa{\tilde{\tau}_{a}(\phi(m))}
      = \tau_{a}(f)\hat{\tau}_{a}(m),
 \end{align*}
 showing that
 $(\A^{n}, \{(\hat{\sigma}_{a}, \hat{\tau}_{a})\}_{a \in I})$ is a
 $\Sigma$-module, and it follows from Proposition~\ref{prop:TM} that
 $(p(\A^{n}), \{(p \circ \hat{\sigma}_{a}, p \circ
 \hat{\tau}_{a})\}_{a\in I})$ is a $\Sigma$-module as well. Moreover, since
 \begin{align*}
   &\hat{\phi} \circ \hat{\sigma}_{a}
   = \hat{\phi} \circ \psi \circ \tilde{\sigma}_{a} \circ \hat{\phi}
     = \tilde{\sigma}_{a} \circ \hat{\phi}\\
   &\hat{\phi} \circ \hat{\tau}_{a}
     = \hat{\phi} \circ \psi \circ \tilde{\tau}_{a} \circ \hat{\phi}
     = \tilde{\tau}_{a} \circ \hat{\phi},
 \end{align*}
 one concludes that
 \begin{align*}
   \hat{\phi}:(p(\A^{n}), \{(p \circ \hat{\sigma}_{a}, p \circ \hat{\tau}_{a})\}_{a\in I})\to
   (M, \{(\tilde{\sigma}_{a}, \tilde{\tau}_{a})\}_{a \in I})   
 \end{align*}
 is a $\Sigma$-module isomorphism. Next, one computes
 \begin{align*}
   &\hat{\sigma}_{a} \circ p
     = \psi \circ \tilde{\sigma}_{a} \circ \phi \circ \psi \circ \phi
     = \psi \circ \tilde{\sigma}_{a} \circ \phi\\
   &p \circ \hat{\sigma}_{a}
     = \psi \circ \phi \circ \psi \circ \tilde{\sigma}_{a} \circ \phi
     = \psi \circ \tilde{\sigma}_{a} \circ \phi
 \end{align*}
 giving $[\hat{\sigma}_{a}, p] = 0.$ In an analogous way, one shows
 that $[\hat{\tau}_{a}, p] = 0.$
\end{proof}

\subsection{$\sigmatau$-$\ast$-algebras and $\Sigma$-$\ast$-bimodules}

\noindent
If $(\A,\{X_a\}_{a\in I})$ is a $\sigmatau$-algebra where $\A$ is a
$\ast$-algebra, we proceed to define $\sigmatau$-$\ast$-algebras,
where one requires certain compatibility conditions with this extra
structure. The introduction of $\ast$-structures on
$\sigmatau$-algebras will become important in the following as we aim
to describe metric structures on $\Sigma$-modules via hermitian forms.

Given a $\ast$-algebra $\A$ and an arbitrary map $\alpha:\A\to\A$, one
defines $\alpha^\ast:\A\to\A$ as $\alpha^\ast(f)=\alpha(f^\ast)^\ast$
for $f\in\A$.  Consequently, given a $\sigmatau$-derivation $X$ one
finds that
\begin{align*}
  X^{\ast}(fg) = \tau^{\ast}(f)X^{\ast}(g) + X^{\ast}(f)\sigma^{\ast}(g)
\end{align*}
for $f, g \in \A$, implying that $X^{\ast}$ is a
$(\tau^{\ast}, \sigma^{\ast})$-derivation. Given $\sigmatau$-algebra
$(\A,\{X_a\}_{a\in I})$, such that $\A$ is a $\ast$-algebra, it is
natural to require that the set $\{X_a\}_{a\in I}$ contains $X_a^\ast$
for all $a\in I$; this leads to the following definition.

\begin{definition}
  Let $\A$ be a $\ast$-algebra. A $\sigmatau$-$\ast$-algebra
  $(\A, \{X_{a}\}_{a \in I}, \iota)$ is a $\sigmatau$-algebra
  $(\A,\{X_a\}_{a\in I})$ together with an involution $\iota: I \to I$
  such that
  \begin{align*}
    X_{a}^{\ast} = X_{\iota(a)}\qand
    \sigma_{\iota(a)} = \tau_{a}^{\ast} 
  \end{align*}
  for all $a \in I$.
\end{definition}

\noindent
The notion of morphism is readily extended to $\sigmatau$-$\ast$-algebras.

\begin{definition}
  A morphism of $\st$-$\ast$-algebras
  \begin{align*}
   (\phi,\psi): \Sigma=(\A,\{X_a\}_{a\in
    I},\iota)\to\Sigma'=(\A',\{Y_\alpha\}_{\alpha\in J},\iota') 
  \end{align*}
  is a morphism of $\st$-algebras
  $(\phi,\psi):(\A,\{X_a\}_{a\in I})\to(\A',\{Y_\alpha\}_{\alpha\in
    J})$ such that $\phi$ is a $\ast$-algebra homomorphism and
  $\psi(X^\ast)=\psi(X)^\ast$ for all $X\in\TSigma$.
\end{definition}

\noindent
Let us now recall the notion of $\ast$-bimodules.

\begin{definition}
  Let $\A$ be a $\ast$-algebra. A $\ast$-bimodule over $\A$ is a
  $\A$-bimodule $M$ together with an involution $\ast : M \to M$ such
  that
  \begin{eqnarray*}
    &(m_{1} + m_{2})^{\ast} = m_{1}^{\ast} + m_{2}^{\ast}& \quad (\lambda m)^{\ast} = \bar{\lambda}m^{\ast}\\
    &(m^{\ast})^{\ast} = m &\quad
                             (fmg)^{\ast} = g^{\ast}m^{\ast}f^{\ast}
  \end{eqnarray*}
  for $m, m_{1}, m_{2} \in M, f, g \in \A$ and $\lambda \in \mathbb{C}.$
\end{definition}

\noindent
For a $\ast$-bimodule $M$ over $\A$ and an arbitrary map $T:M\to M$,
one sets
\begin{align*}
  T^\ast(m) = \paraa{T(m^\ast)}^\ast.
\end{align*}
Note that, given a $\sigmatau$-$\ast$-algebra
$\Sigma=(\A,\{X_a\}_{a\in I},\iota)$ and a $\Sigma$-bimodule
$(M, \{(\sigmah_{a}, \tauh_{a})\}_{a \in I})$ such that $M$ is a
$\ast$-bimodule, it follows that
\begin{align*}
  &\sigmah_{a}^{\ast}(fm) = \sigma_{a}^{\ast}(f)\sigmah_{a}^{\ast}(m)\\
  &\tauh_{a}^{\ast}(fm) = \tau_{a}^{\ast}(f)\tauh_{a}^{\ast}(m).
\end{align*}
In correspondence with $\sigmatau$-$\ast$-algebras, let us introduce
$\Sigma$-$\ast$-modules.

\begin{definition}
  Let $\Sigma$ be a $(\sigma, \tau)$-$\ast$-algebra. A
  $\Sigma$-$\ast$-module is a $\Sigma$-module
  $(M, \{(\sigmah_{a}, \tauh_{a})\}_{a \in I})$ such that $M$ is a
  $\ast$-bimodule and $\sigmah_{\iota(a)}=\tauh_a^\ast$
  for all $a \in I$.
\end{definition}

\begin{example}
  Let $\A$ be a $\ast$-algebra and $\A^{n}$ be a free left $\A$-module
  with basis $\{e_i\}_{i=1}^n$ together with the bimodule structure
  $(m^ie_i)f = (m^if)e_i$ for $f\in\A$ and $m^ie_i\in \A^n$.  Setting
  $m^{\ast} = m^{i \ast}e_{i}$ for $m = m^{i}e_{i},$ one has
  \begin{eqnarray*}
    &&(m_{1} + m
       _{2})^{\ast} = (m_{1}^{i} + m_{2}^{ i})^{\ast}e_{i} = (m_{1}^{i \ast} + m_{2}^{i \ast})e_{i} = m_{1}^{\ast} + m_{2}^{\ast}\\
    &&(m^{\ast})^{\ast} = (m^{i \ast})^{\ast}e_{i} = m\\
    &&(fmg)^{\ast} = (fm^{i}g)^{\ast}e_{i} = (g^{\ast} m^{i \ast} f^{\ast})e_{i} = g^{\ast} m^{\ast} f^{\ast},
  \end{eqnarray*}
  showing that $\A^{n}$ is a $\ast$-bimodule. Given a
  $(\sigma, \tau)$-$\ast$-algebra
  $\Sigma = (\A, \{X_{a}\}_{a \in I},\iota)$ and the $\Sigma$-bimodule
  $(\A^{n}, \{(\sigmah_{a}^{0}, \tauh_{a}^{0})\}_{a \in I})$, one
  finds that
  \begin{eqnarray*}
    &&(\sigmah_{a}^{0})^{\ast}(m) = \sigmah_{a}^{0}(m^{\ast})^{\ast} = (\sigmah_{a}^{0}(m^{i \ast}e_{i}))^{\ast} = (\sigma_{a}(m^{i \ast}))^{\ast}e_{i} = \sigma_{a}^{\ast}(m^{i})e_{i} \\
    &&= \tau_{\iota(a)}(m^{i})e_{i} = \tauh_{\iota(a)}^{0}(m).
  \end{eqnarray*}
  showing that $(\sigmah_{a}^{0})^{ \ast} = \tauh_{\iota(a)}^{0}$. Hence,
  $(\A^{n}, \{(\sigmah_{a}^{0}, \tauh_{a}^{0})\}_{a \in I})$ is a
  $\Sigma$-$\ast$-bimodule.
\end{example}

\noindent
As an illustration of the above concepts, let us construct a
$\st$-$\ast$-algebra, over the quantum 3-sphere.  The quantum 3-sphere
$S^{3}_{q}$ is a $\ast$-algebra generated by elements $a$ and $c$
satisfying
\begin{align*}
  &ac = q ca, \quad c^{\ast} a^{\ast} = q a^{\ast}c^{\ast} \\
  &ac^{\ast} = q c^{\ast} a, \quad ca^{\ast} = qa^{\ast}c\\
  &cc^{\ast} = c^{\ast}c, \quad a^{\ast}a + c^{\ast}c = aa^{\ast} + q^{2}cc^{\ast} = \mid.
\end{align*}
Moreover, the quantum universal enveloping
algebra $\U_{q}(su(2))$ is a Hopf-$\ast$-algebra generated by
$E, F, K, K^{-1}$ satisfying the following conditions,
\begin{align*}
  &K^{\pm 1}E = q^{\pm 1}EK^{\pm 1}, \quad K^{\pm 1} F = q^{\mp 1}FK^{\pm 1}\\
  & [E, F] = \dfrac{ K - K^{-1}}{q - q^{-2}}
\end{align*}
from which one defines
\begin{align*}
  \Xp=  \sqrt{q}EK\qquad\Xm = \frac{1}{\sqrt{q}}FK\qquad \Xz = \dfrac{1 - K^{4}}{1 - q^{-2}}.
\end{align*}
The comultiplication, antipode and counit are given by
\begin{align*}
  &\Delta(K^{\pm 1}) = K^{\pm 1} \otimes K^{\pm 1}, \quad \Delta(E) = E \otimes K + K^{-1} \otimes E\\
  &\Delta (F) = F \otimes K^{-1} + K \otimes F,\\
  &S(K) = K^{-1}, \quad S(E) = -qE, \quad S(F) = -q^{-1}F,\\
  &\epsilon(K) = 1, \quad \epsilon(E) = 0,\quad \epsilon(F) = 0,
\end{align*}
and we recall that there is a unique bilinear pairing between $\Uqsu$ and
$\Sthreeq$ given by
\begin{alignat*}{2}
&\angles{K^{\pm 1}, a} = q^{\mp\, 1/2} &\qquad 
&\angles{K^{\pm 1}, a^*} = q^{\mp\, 1/2}\\
&\angles{E,c} = 1 &  &\angles{F,c^*} = -q^{-1}, 
\end{alignat*}
with the remaining pairings being zero. The pairing
induces a $\Uqsu$-bimodule structure on $\Sthreeq$ given by
\begin{align}\label{actions}
  h\lt f = \one{f}\angles{h,\two{f}}\qand
  f\rt h = \angles{h,\one{f}}\two{f}
\end{align}
for $h\in\Uqsu$ and $f\in\Sthreeq$ using Sweedler's notation
$\Delta(f)=\sum\one{f}\otimes\two{f}$. As it will be relevant in the
following, let us list the (left) action of $K$ on $S^3_q$:
\begin{equation}\label{eq:K.left.action}
  \begin{split}
    &K\lt a = q^{-\frac{1}{2}} a\qquad
    K\lt a^\ast = q^{\frac{1}{2}}a^\ast\\
    &K\lt c = q^{-\frac{1}{2}}c\qquad
    K\lt c^\ast = q^{\frac{1}{2}}c^\ast,    
  \end{split}
\end{equation}
and the structure of $\U_q(su(2))$ implies that $K$ is an algebra
endomorphism. In the following, we shall simply denote the left action
of $K$ by $K(f)$ for $f\in S^3_q$.

Now, let us proceed to construct a $\sigmatau$-$\ast$-algebra based on
$S^3_q$. It follows from the action of $\U_q(su(2))$ that
\begin{align*}
  &\Xp(fg) = f\Xp(g) + \Xp(f)K^{2}(g)\\
  &\Xm(fg) = f\Xm(g) + \Xm(f)K^{2}(g)\\
  &\Xz(fg) = f\Xz(g) + \Xz K^{4}(g)
\end{align*}
as well as
\begin{align*}
  &\Xp^\ast(f) \equiv \paraa{\Xp(f^\ast)}^\ast = -K^{-2}\paraa{\Xm(f)}\\
  &\Xm^\ast(f) \equiv \paraa{\Xm(f^\ast)}^\ast = -K^{-2}\paraa{\Xp(f)}\\
  &\Xz^\ast(f) \equiv \paraa{\Xz(f^\ast)}^\ast = -K^{-4}\paraa{\Xz(f)}.
\end{align*}
Hence, defining
\begin{align*}
  Y_1 &= iK^{-1}\Xp + iK^{-1}\Xm\\
  Y_2 &= K^{-1}\Xm - K^{-1}\Xp\\
  Y_3 &= iK^{-2}\Xz
\end{align*}
one finds that $Y_a^\ast=Y_a$ for $a=1,2,3$, and
\begin{align*}
  &Y_1(fg) = K^{-1}(f)Y_1(g) + Y_1(f)K(g)\\
  &Y_2(fg) = K^{-1}(f)Y_2(g) + Y_2(f)K(g)\\
  &Y_3(fg) = K^{-2}(f)Y_3(g) + Y_3(f)K^2(g)
\end{align*}
implying that $Y_1,Y_2,Y_3$ are $\sigmatau$-derivations with
\begin{align*}
  \sigma_1=K^{-1}\quad \sigma_2=K^{-1}\quad \sigma_3 = K^{-2}\qquad
  \tau_1=K\quad \tau_2=K\quad \tau_3=K^2.
\end{align*}
Hence, $\Sigma=(S^3_q,\{Y_1,Y_2,Y_3\})$ is a $\st$-algebra. Next, we note that
\begin{align*}
  \sigma_1^\ast(f) = \paraa{K^{-1}(f^\ast)}^\ast = K(f) = \tau_1(f),
\end{align*}
and an analogous computation for $\sigma_2,\sigma_3$ shows that
$\sigma_a^\ast = \tau_a$ for $a=1,2,3$. Finally, defining
$\iota:\{1,2,3\}\to\{1,2,3\}$ as $\iota(a)=a$, it follows that
\begin{align*}
  \Sigma^\ast = (\Sthreeq,\{Y_1,Y_2,Y_3\},\iota)
\end{align*}
is a $\st$-$\ast$-algebra since
$\sigma_{\iota(a)} = \sigma_a = \tau_a^\ast$ for $a=1,2,3$.  Next, let
us construct a $\Sigma$-$\ast$-bimodule over $\Sthreeq$.

As is well known, the standard module $\Omega^{1}(S^{3}_{q})$ of
differential forms over $\Sthreeq$ is freely generated by
$\omega_{+}, \omega_{-}, \omega_{z}$ defined as
\begin{align*}
  \omega_{+} = \omega_{1} = adc -q c da, \quad
  \omega_{-} = \omega_{2} = c^{\ast} da  - q a^{\ast}d c^{\ast} \quad
  \omega_{z} = \omega_{3} = a^{\ast}da + c^{\ast} dc
\end{align*}
where the differential $d : S^{3}_{q} \to \Omega^{1}S^{3}_{q}$ is given by 
\begin{eqnarray*}
  df = X_{+}(f)\omega_{+} + X_{-}(f)\omega_{-} + X_{z}(f)\omega_{z}
\end{eqnarray*}
for $f \in S^{3}_{q}$. Moreover $\Omega^{1}(S^{3}_{q})$ is a
$\ast$-bimodule with
$\omega_{+}^{\ast} = -\omega_{-}, \omega_{z}^{\ast} = -\omega_{z}$
together with the following bimodule relations
\begin{equation}\label{bimodule str}
  \begin{split}
    &\omega_{z}a = q^{-2}a\omega_{z} \quad \omega_{z} a^{\ast} = q^{2}a^{\ast}\omega_{z}
    \quad \omega_{\pm} a = q^{-1}a \omega_{\pm}, \quad \omega_{\pm} a^{\ast} = q a^{\ast}\omega_{\pm}\\
    &\omega_{z}c = q^{-2}c\omega_{z} \quad \omega_{z}c^{\ast} = q^{2}c^{\ast}\omega_{z}
    \quad \omega_{\pm} c = q^{-1} c \omega_{\pm} \quad \omega_{\pm} c^{\ast} = q c^{\ast}\omega_{\pm}.
  \end{split}
\end{equation}
The bimodule relations can be conveniently written in the following way.
\begin{lemma}\label{lemma:bimodule.K.commutation}
  Let $n_1=n_2=2$ and $n_3=4$. Then
  \begin{align}\label{eq:omega.f.K}
    \omega_a f = K^{n_a}(f)\omega_a
  \end{align}
  for $f\in S^3_q$ and $a=1,2,3$.
\end{lemma}

\begin{proof}
  By using \eqref{eq:K.left.action} and \eqref{bimodule str}, it is easy to check that
  \begin{align*}
    \omega_a f = K^{n_a}(f)\omega_a
  \end{align*}
  for $f\in\{a,a^\ast,c,c^\ast\}$ and $a=1,2,3$. Next, one notes that
  if \eqref{eq:omega.f.K} holds for $f_1,f_2\in S^3_q$ then
  \begin{align*}
    \omega_a(f_1f_2)
    &= (\omega_af_1)f_2 = \paraa{K^{n_a}(f_1)\omega_a}f_2
    =K^{n_a}(f_1)(\omega_af_2) = K^{n_a}(f_1)K^{n_a}(f_2)\omega_a\\
    &= K^{n_a}(f_1f_2)\omega_a,
  \end{align*}
  from which it immediately follows that \eqref{eq:omega.f.K} holds
  for all monomials
  $a^{k_1}(a^\ast)^{k_2}c^{k_3}(c^\ast)^{k_4}$. Since such monomials
  generate $S^3_q$ as a vector space, \eqref{eq:omega.f.K} holds for all $f\in S^3_q$.
\end{proof}

\noindent
For our purposes, a slightly more convenient basis of $\Omega^1(S^3_q)$ is given by
\begin{align*}
  \eta_1 = i(\omegap + \omegam)\qquad
  \eta_2 = \omegam - \omegap \qquad
  \eta_3 = i\omegaz
\end{align*}
satisfying $\eta_a^\ast = \eta_a$ for $a=1,2,3$. It is clear that
Lemma~\ref{lemma:bimodule.K.commutation} implies that
\begin{align}
  \eta_af = K^{n_a}(f)\eta_a
\end{align}
for $f\in S^3_q$ and $a=1,2,3$.  Furthermore, we extend the action of
$K$ to $\Omega^1(S^3_q)$ by setting
\begin{align*}
  \Kh(m^a\eta_a) = K(m^a)\eta_a,
\end{align*}
and one finds that for $m=m^a\eta_a\in\Omega^1(S^3_q)$ and $f,g\in S^3_q$
\begin{align*}
  \Kh(fmg)
  &= \Kh\paraa{f(m^a\eta_a)g} = \Kh\paraa{fm^aK^{n_a}(g)\eta_a}
    = K(fm^a)K^{n_a}\paraa{K(g)}\eta_a\\
  &= K(fm^a)\eta_aK(g) = K(f)\Kh(m)K(g). 
\end{align*}
Moreover,
\begin{align*}
  \Kh^\ast(m^a\eta_a)
  &= \paraa{\Kh((m^a\eta_a)^\ast)}^\ast
    =\paraa{\Kh(\eta_a(m^a)^\ast)}^\ast
    =\paraa{\Kh(K^{n_a}((m_a)^\ast)\eta_a)}^\ast \\
  &=\paraa{K^{n_a+1}\paraa{(m^a)^\ast}\eta_a}^\ast
    =\eta_aK^{n_a+1}\paraa{(m^a)^\ast}^\ast
    =\eta_aK^{-n_a-1}(m^a)\\
  &= K^{-n_a-1+n_a}(m^a)\eta_a
    = K^{-1}(m^a)\eta_a,
\end{align*}
by using that $K(f^\ast)^\ast=K^{-1}(f)$, which implies that
$\Kh^\ast=\Kh^{-1}$. Finally, setting
\begin{align*}
  &\sigmah_1= \Kh^{-1}\quad \sigmah_2=\Kh^{-1}\quad \sigmah_3=\Kh^{-2}\qquad
  \tauh_1 = \Kh\quad \tauh_2 = \Kh\quad\tauh_3 = \Kh^2
\end{align*}
one easily checks that
$(\Omega^1(S^3_q),\{(\sigmah_a,\tauh_a)\}_{a=1}^3)$ is a
$\Sigma$-$\ast$-bimodule since $\Kh^\ast=\Kh^{-1}$ and
$\Kh(fmg)=K(f)\Kh(m)K(g)$ for all $f,g\in S^3_q$ and $m\in \Omega^1(S^3_q)$.

\section{$\sigmatau$-connections}

\noindent
In this section, we will introduce connections on
$\Sigma$-modules. For a $\sigmatau$-algebra $\Sigma$ and a
$\Sigma$-module $M$, these connections give covariant derivatives on
$M$, in the direction of a $\sigmatau$-derivation. Now, these
derivatives will satisfy a twisted Leibniz' rule, and our guiding
principle is to consider a $\sigmatau$-derivation $X$ as a covariant
derivative on the module consisting of the algebra itself;
i.e. $\nabla_X(f)=X(f)$ for $f\in\A$. It then follows that
\begin{align*}
  \nabla_X(fg) = X(fg) = \sigma(f)X(g)+X(f)\tau(g)
  =\sigma(f)\nabla_X(g) + X(f)\tau(g),
\end{align*}
which we generalize to arbitrary $\Sigma$-modules in the following way.

\begin{definition}\label{def:left.sigmatau.connection}
  Let $\Sigma=(\A,\{X_a\}_{a\in I})$ be a $\sigmatau$-algebra and let
  $(M,\{(\sigmah_a,\tauh_a)\}_{a\in I})$ be a left $\Sigma$-module. A
  left \emph{$\sigmatau$-connection on $M$} is a map
  $\nabla:\TSigma\times M\to M$ satisfying
  \begin{align*}
    &\nabla_{X}(\lambda_1m_{1} + \lambda_2m_{1}) = \lambda_1\nabla_{X}m_{1} + \lambda_2\nabla_{X}m_{2}\\
    &\nabla_{\lambda_1X + \lambda_2Y}m = \lambda_1\nabla_{X}m + \lambda_2\nabla_{Y}m\\
    &\nabla_{X_{a}}(fm) = \sigma_{a}(f)\nabla_{X_{a}}m + X_{a}(f)\hat{\tau}_{a}(m)
  \end{align*}
  for all $X,Y\in\TSigma$, $m, m_{1}, m_{2} \in M$, $\lambda_1,\lambda_2\in\complex,$ $f \in \A$  and $a\in I$.
\end{definition}

\noindent
For completeness, let us spell out the definition of right
$\sigmatau$-connections, as well as $\sigmatau$-bimodule connections.

\begin{definition}\label{def:right.sigmatau.connection}
  Let $\Sigma=(\A,\{X_a\}_{a\in I})$ be a $\sigmatau$-algebra and let
  $(M,\{(\sigmah_a,\tauh_a)\}_{a\in I})$ be a right $\Sigma$-module. A
  right \emph{$\sigmatau$-connection on $M$} is a map
  $\nabla:\TSigma\times M\to M$ satisfying
  \begin{align*}
    &\nabla_{X}(\lambda_1m_{1} + \lambda_2m_{1}) = \lambda_1\nabla_{X}m_{1} + \lambda_2\nabla_{X}m_{2}\\
    &\nabla_{\lambda_1X + \lambda_2Y}m = \lambda_1\nabla_{X}m + \lambda_2\nabla_{Y}m\\
    &\nabla_{X_{a}}(mf) = \sigmah_{a}(m)X_a(f) + \nabla_{X_a}(m)\tau_{a}(f)
  \end{align*}
  for all $X,Y\in\TSigma$, $m, m_{1}, m_{2} \in M$, $\lambda_1,\lambda_2\in\complex,$ $f \in \A$  and $a\in I$.
\end{definition}

\begin{definition}
  Let $\Sigma=(\A,\{X_a\}_{a\in I})$ be a $\sigmatau$-algebra and let
  $(M,\{(\sigmah_a,\tauh_a)\}_{a\in I})$ be a $\Sigma$-bimodule. If
  $\nabla$ is both a left and right $\sigmatau$-connection on $M$,
  then $\nabla$ is called a $\sigmatau$-bimodule connection.
\end{definition}

\begin{example}\label{ex:nablazero.def}
  Let $\Sigma = (\A,\{X_{a}\}_{a \in I})$ be a
  $(\sigma, \tau)$-algebra and let
  $(\A^{n}, \{(\hat{\sigma}_{a}^{0}, \hat{\tau}_{a}^{0})\}_{a \in I})$
  be the free $\Sigma$-bimodule as defined in
  Example~\ref{ex:free.module.sigma.structure}. Introducing
  $\nabla^{0} : T\Sigma \times \A^{n} \to \A^{n}$, defined by
  $\nabla^{0}_{X}(m) = X(m^{i})e_{i}$ for $m = m^{i}e_{i} \in \A^{n}$,
  one can easily see that
  \begin{eqnarray*}
    &&\nabla^{0}_{X}(\lambda_1m_{1} + \lambda_2m_{2}) = \lambda_1\nabla^{0}_{X}(m_{1})
       + \lambda_2\nabla^{0}_{X}(m_{2}) \\
    &&\nabla^{0}_{\lambda_1X + \lambda_2Y}(m) = \lambda_1\nabla^{0}_{X}(m) + \lambda_2\nabla^{0}_{Y}(m),
  \end{eqnarray*}
  for $m, m_{1}, m_{2} \in \A^{n}$, $\lambda_1,\lambda_2\in\complex$ and $X,Y\in\TSigma$. Moreover, for 
  $a\in I$, $f \in \A$ and $m \in \A^{n},$ one finds that
  \begin{eqnarray*}
    &\nabla^{0}_{X_{a}}(fm) = \sigma_{a}(f)\nabla^{0}_{X}(m) + X_{a}(f)\hat{\tau}^{0}_{a}(m)\\
    &\nabla^{0}_{X_a}(mf) = (\nabla^0_{X_a}m)\tau_a(f)+\sigmah^0(m)X_a(f)
  \end{eqnarray*}
  Hence, $\nabla^{0}$ is a $\sigmatau$-bimodule connection on the $\Sigma$-bimodule
  $(\A^{n}, \{(\hat{\sigma}_{a}^{0}, \hat{\tau}_{a}^{0})\}_{a \in I})$.
\end{example}

\noindent
In a slightly more general setting than
Example~\ref{ex:nablazero.def}, let us formulate the existence of
$\sigmatau$-connections on free modules in the following way.

\begin{proposition} \label{sigma tau conn arb}
  $\Sigma=(\A,\{X_a\}_{a\in I})$ be a $\sigmatau$-algebra and let
  $(\A^{n}, \{(\hat{\sigma}_{a}, \hat{\tau}_{a})\}_{a \in I})$ be a
  free left $\Sigma$-module. Given $\Gamma_{ai}^{j} \in \A$, for $a\in I$
  and $1\leq i,j\leq n$, there exists a left $(\sigma, \tau)$-connection on
  $(\A^{n}, \{(\sigmah_{a}, \tauh_{a})\}_{a \in I})$ such that
  \begin{align*}
    \nabla_{X_{a}}(e_{i}) = \Gamma_{ai}^{j}e_{j}
  \end{align*}
  for $a \in I$ and $i,j = 1, \ldots, n$.
\end{proposition}

\begin{proof}
One defines $\nabla_{X_a}(e_{i}) = \Gamma_{ai}^{j}e_{j}$ and 
\begin{eqnarray*}
&&\nabla_{X_{a}}(m^{i}e_{i}) = \sigma_{a}(m^{i})\nabla_{X_{a}}(e_{i}) + X_{a}(m^{i})\tauh_{a}(e_{i}),\\
&& \nabla_{\lambda X_{a} + \mu X_{b}}(e_{i}) = \lambda\nabla_{X_{a}}(e_{i}) + \mu\nabla_{X_{b}}(e_{i})
\end{eqnarray*}
for $m = m^{i}e_{i} \in \A^{n},$ $a\in I$ and
$\lambda, \mu\in \complex$.  It follows that $\nabla$ satisfies the
linearity properties of a $\sigmatau$-connection, and moreover
\begin{eqnarray*}
  \nabla_{X_{a}}(fm) &=& \nabla_{X_{a}}(fm^{i}e_{i})
                         = \sigma_{a}(fm^{i})\nabla_{X_{a}}(e_{i}) + X_{a}(fm^{i})\tauh_{a}(e_{i})\\
&=& \sigma_{a}(f)\sigma_{a}(m^{i})\nabla_{X_{a}}(e_{i}) + (\sigma_{a}(f)X_{a}(m^{i}) + X_{a}(f)\tau_{a}(m^{i}))\tauh_{a}(e_{i})\\
&=& \sigma_{a}(f)(\sigma_{a}(m^{i})\nabla_{X_{a}}(e_{i}) + X_{a}(m^{i})\tauh_{a}(e_{i})) + X_{a}(f)\tau_{a}(m^{i})\tauh_{a}(e_{i})\\
&=& \sigma_{a}(f)\nabla_{X_{a}}(m) + X_{a}(f)\tauh_{a}(m),
\end{eqnarray*}
for $f \in \A$ and $m \in \A^{n}$, showing that $\nabla$ is indeed a left $\sigmatau$-connection.
\end{proof}

\noindent
Given a $\sigmatau$-connection $\nabla$ on a $\Sigma$-module
$(M, \{(\hat{\sigma}_{a}, \hat{\tau}_{a})\}_{a \in I})$ together
with an endomorphism $T:M\to M$, one can construct a
$\sigmatau$-connection on the image of $T$ as follows.

\begin{proposition}\label{con on proj mod}
  Let $(M,\{(\hat{\sigma}_{a}, \hat{\tau}_{a})\}_{a \in I})$ be a left
  $\Sigma$-module and let $\nabla$ be a left
  $(\sigma, \tau)$-connection on
  $(M,\{(\hat{\sigma}_{a}, \hat{\tau}_{a})\}_{a \in I}).$ If
  $T: M \to M$ is an endomorphism then $ \tilde{\nabla} = T \circ \nabla$
  is a left $(\sigma, \tau)$-connection on
  $(T(M), \{(T \circ \hat{\sigma}_{a}, T\circ \hat{\tau}_{a})\}_{a \in
    I}).$
\end{proposition}
\begin{proof}
  By Proposition~\ref{prop:TM}
  $(T(M), \{(T \circ \sigmah_{a}, T \circ \tauh_{a})\}_{a \in I})$ is
  a $\Sigma$-module.  Furthermore, it is clear that $\nablat$
  satisfies the linearity properties in
  Definition~\ref{def:left.sigmatau.connection} since $T$ is a module
  endomorphism. For $f \in A,$ $m \in M$ and $X_{a} \in T\Sigma,$ one
  obtains
  \begin{align*}
    \tilde{\nabla}_{X_{a}}(fm)
    &= T(\nabla_{X_{a}}(fm))
      = T(\sigma_{a}(f)\nabla_{X_{a}}m) + T(X_{a}(f)\hat{\tau}_{a}(m))\\
    &= \sigma_{a}(f)T(\nabla_{X_{a}}m) + X_{a}(f)T(\hat{\tau}_{a}(m))
    = \sigma_{a}(f)\tilde{\nabla}_{X_{a}}m + X_{a}(f)(T \circ \tauh_{a})(m),
  \end{align*}
  showing that $\nablat$ is a left $\sigmatau$-connection on $(T(M), \{(T \circ \sigmah_{a}, T \circ \tauh_{a})\}_{a \in I})$.  
\end{proof}

\noindent
In particular, one can apply Proposition~\ref{con on proj mod} to the
case when $M=\A^n$ and $T$ is a projection, to obtain the following
result.

\begin{theorem}
  There exists a $\st$-connection on every projective $\Sigma$-module.
\end{theorem}

\begin{proof}
  Let $(M,\{(\tilde{\sigma}_{a}, \tilde{\tau}_{a})\}_{a \in I})$ be a
  $\Sigma$-module where such that $M$ is a projective $\A$-module. By
  Proposition \ref{proj mod iso} there exists a projection
  $p :\A^{n} \to \A^{n}$ such that
  \begin{align}\label{eq:p.proj.iso}
    (p(\A^{n}), \{(p \circ \hat{\sigma}_{a}, p \circ \hat{\tau}_{a})\}_{a \in I}) \simeq (M, \{(\tilde{\sigma}_{a}, \tilde{\tau}_{a})\}_{a \in I}).
  \end{align}
  Moreover, by Proposition \ref{sigma tau conn arb}, there exists a
  $(\sigma, \tau)$-connection $\nabla$ on
  $(\A^{n}, \{(\sigmah_{a}, \tauh_{a})\}_{a \in I})$. Defining
  $\tilde{\nabla} = p \circ \nabla$, it follows from
  Proposition~\ref{con on proj mod} that $\tilde{\nabla}$ is a
  $(\sigma, \tau)$-connection on
  $(p(\A^{n}), \{(p \circ \hat{\sigma}_{a}, p \circ
  \hat{\tau}_{a})\}_{a \in I}).$ By the isomorphism in
  \eqref{eq:p.proj.iso}, it is clear that there exists a
  $\sigmatau$-connection on the projective $\Sigma$-module
  $(M, \{(\tilde{\sigma}_{a}, \tilde{\tau}_{a})\}_{a \in I})$.
\end{proof}

\noindent
If $(M, \{(\sigmah_{a}, \tauh_{a})\}_{a \in I})$ is a
$\Sigma$-$\ast$-bimodule, one can demand that connections are
compatible with the $\ast$-structure on the module.

\begin{definition}
  Let $(M, \{(\sigmah_{a}, \tauh_{a})\}_{a \in I})$ be a
  $\Sigma$-$\ast$-bimodule. A $\sigmatau$-$\ast$-bimodule connection
  on $(M, \{(\sigmah_{a}, \tauh_{a})\}_{a \in I})$ is a
  $\sigmatau$-bimodule connection $\nabla$ such that
  \begin{align}\label{eq:sigmatau.star.connection}
    (\nabla_{X_a}m)^\ast=\nabla_{X_a^\ast}m^\ast  
  \end{align}
  for all $a\in I$ and $m\in M$.
\end{definition}

\noindent
In the above definition, one assumes that $\nabla$ is a
$\sigmatau$-bimodule connection. However, given
\eqref{eq:sigmatau.star.connection}, a left $\sigmatau$-connection is
automatically a $\sigmatau$-bimodule connection.

\begin{proposition}\label{prop:star.connection.bimodule}
  Assume that $(M,\{(\sigmah_{a}, \tauh_{a})\}_{a \in I})$ is a
  $\Sigma$-$\ast$-bimodule and let $\nabla$ be a left
  $\sigmatau$-connection on
  $(M,\{(\sigmah_{a}, \tauh_{a})\}_{a \in I})$. If
  $(\nabla_{X_a}(m))^{\ast} = \nabla_{X_a^{\ast}}(m^{\ast})$ for all
  $a\in I$ and $m \in M$ then $\nabla$ is a
  $(\sigma, \tau)$-$\ast$-bimodule connection.
\end{proposition}
\begin{proof}
  To prove that $\nabla$ is a $\sigmatau$-$\ast$-bimodule connection, one needs to show that
  \begin{align}\label{eq:star.bimodule.right}
    \nabla_{X_a}(mf) =
    (\nabla_{X_a}m)\tau_a(f) + \sigmah_a(m)X_a(f)
  \end{align}
  for all $a\in I$, $m\in M$ and $f\in\A$. Using that $\nabla$ is a
  left $\sigmatau$-connection, together with
  $(\nabla_{X_a}(m))^{\ast} = \nabla_{X_a^{\ast}}(m^{\ast})$, one
  finds that
  \begin{align*}
    \nabla_{X_a}(mf)
    &= \nabla_{X_a}\paraa{(f^\ast m^\ast)^\ast}
      = \paraa{\nabla_{X_a^\ast}f^\ast m^\ast}^\ast
      = \paraa{\nabla_{X_{\iota(a)}}f^\ast m^\ast}^\ast\\
    &= \parab{
      \sigma_{\iota(a)}(f^\ast)\nabla_{X_a^\ast}m^\ast
      + X_{a}^\ast(f^\ast)\tauh_{\iota(a)}(m^\ast)
    }^\ast\\
    &= \paraa{\nabla_{X_a^\ast}m^\ast}^\ast\sigma_{\iota(a)}^\ast(f)
      +\tauh^\ast_{\iota(a)}(m)X_a(f)\\
    &= (\nabla_{X_a}m)\tau_a(f) + \tauh^\ast_{\iota(a)}(m)X_a(f).
  \end{align*}
  Since $(M,\{(\sigmah_{a}, \tauh_{a})\}_{a \in I})$ is a
  $\Sigma$-$\ast$-bimodule it holds that
  $\tauh^\ast_{\iota(a)}=\sigmah_a$, from which we conclude that
  \eqref{eq:star.bimodule.right} is satisfied. Hence, $\nabla$ is a
  $\sigmatau$-$\ast$-bimodule connection.  
\end{proof}

\subsection{Metric compatibility of $\sigmatau$-connections}

\noindent
From a differential geometric point of view, the Riemannian metric
gives an inner product of two vector fields at each point. In
noncommutative geometry, the role of the metric is often played by a
hermitian form acting on pairs of elements of a module. The symmetry
of the Riemannian metric is replaced by a skew-symmetry with respect
to the $\ast$-structure of the algebra, which is more adapted to the
setting of noncommutative algebras.

After recalling the concept of a hermitian form, we will proceed to
define compatibility between hermitian forms and
$\sigmatau$-connections. In the following, we will for simplicity focus on
left modules but analogous results may be obtained for right modules.

\begin{definition}
  Let $\A$ be a $\ast$-algebra and let $M$ be a left $\A$-module. A
  hermitian form on $M$ is a map $h: M\times M \to \A$ such
  that
  \begin{align*}
    &h(\lambda_1 m_{1} + \lambda_2m_{2}, m_{3}) = \lambda_1 h(m_{1}, m_{3}) + \lambda_2h(m_{2}, m_{3})\\
    &h(fm_{1}, m_{2}) = fh(m_{1}, m_{2})\\
    &h(m_{1}, m_{2})^{\ast} = h(m_{2}, m_{1})
  \end{align*}
  for $m_{1}, m_{2}, m_{3}  \in M$, $\lambda_1,\lambda_2\in\complex$ and $f \in \A$.
\end{definition}

\noindent Given an hermitian form $h$ on the module $M$, one may
define $\hh:M\to M^\ast$ (where $M^\ast$ denotes the dual module of
$M$) as
\begin{align*}
  \hh(m_1)(m_2) = h(m_1,m_2)
\end{align*}
for $m_1,m_2\in M$. In differential geometry, $\hh$ corresponds to the
isomorphism between vector fields and differential forms induced by
the metric on a Riemannian manifold. However, in the general setting
above, there is no need for $\hh$ to be a bijection.

\begin{definition}
  Let $h$ be a hermitian form on $M$. If $\hh:M\to M^\ast$ is a
  bijection then we say that $h$ is invertible.
\end{definition}

\noindent
For a free module $\A^n$ with basis $e_1,\ldots,e_n$, a hermitian form
$h$ is given in terms of its components $h_{ij}=h(e_i,e_j)$ as
\begin{align*}
  h(m_1,m_2) = m_1^ih_{ij}(m_2^j)^\ast
\end{align*}
for $m_1=m_1^ie_i$ and $m_2=m_2^ie_i$.  A natural class of hermitian forms for
$\sigmatau$-$\ast$-algebras consists of those that are invariant
under the action of $\sigma_a$ and $\tau_a$ for $a\in I$; for a free
module over a $\sigmatau$ $\ast$-algebra this amounts to
\begin{align*}
  \sigma_a(h_{ij})=h_{ij}\qand
  \tau_a(h_{ij})=h_{ij}
\end{align*}
for $a\in I$, implying that
\begin{align*}
  \sigma_a\paraa{h(m_1,m_2)}
  &= \sigma_a(m_1^i)h_{ij}\sigma_a^\ast(m_2^j)^\ast
  = h\paraa{\sigmahz_a(m_1),\tauhz_{\iota(a)}(m_2)}
\end{align*}
for $m_1,m_2\in M$. For general $\Sigma$-modules we make the following definition.

\begin{definition}\label{def l-r hermitian form}
  Let $\Sigma=(\A,\{X_a\}_{a\in I},\iota)$ be a $\sigmatau$-$\ast$-algebra
  and let $(M,\{(\hat{\sigma}_{a}, \hat{\tau}_{a})\}_{a \in I})$ be a
  left $\Sigma$-module. A hermitian form $h$ on $M$ is called 
  \emph{$(\sigma ,\tau)$-invariant} if 
  \begin{align*}
    &\sigma_{a}\paraa{ h(m_{1},m_{2})} = h\paraa{\sigmah_{a}(m_{1}), \tauh_{\iota(a)}(m_{2})} \\
    &\tau_{a}\paraa{h(m_{1}, m_{2})} = h\paraa{\tauh_{a}(m_{1}), \sigmah_{\iota(a)}(m_{1})}
  \end{align*}
  for $m_{1}, m_{2} \in M$.
\end{definition}

\noindent
Note that for a $\Sigma$-$\ast$-bimodule one can write the
condition for a hermitian form to be invariant as
\begin{align*}
  &\sigma_a\paraa{h(m_1,m_2)}=h\paraa{\sigmah_a(m_1),\sigmah_a^\ast(m_2)}\\
  &\tau_a\paraa{h(m_1,m_2)}=h\paraa{\tauh_a(m_1),\tauh_a^\ast(m_2)},
\end{align*}
since $\sigmah_a^\ast = \tauh_{\iota(a)}$.
Let us now approach a definition of what it means for a
connection to be compatible with a hermitian form on a module. As a
motivation, we consider the case of a free module and a hermitian form
such $X_a(h_{ij})=0$.

To this end, let $\Sigma=(\A,\{X_a\}_{a\in I},\iota)$ be a $\sigmatau$
$\ast$-algebra and let $\nabla^0$ be the $\sigmatau$-connection on the
free $\Sigma$-module $(\A^n,\{\sigmahz_a,\tauhz_a\})$ as introduced in
Example~\ref{ex:nablazero.def}. For a $\sigmatau$-invariant hermitian form
$h:M\times M\to\A$, assuming $X_a(h_{ij})=0$ for $a\in I$, one finds that
\begin{align*}
  X_a\paraa{h(m_1,m_2)}
  &= X_a\paraa{m_1^ih_{ij}(m_2^j)^\ast}
    = \sigma_a(m_1^i)X_a\paraa{h_{ij}(m_2^j)^\ast}
    + X_a(m_1^i)\tau_a\paraa{h_{ij}(m_2^j)^\ast}\\
  &= \sigma_a(m_1^i)h_{ij}X_a\paraa{(m_2^j)^\ast}
    +X_a(m_1^i)h_{ij}\tau_a\paraa{(m_2^j)^\ast}\\
  &=  \sigma_a(m_1^i)h_{ij}X_{a}^\ast(m_2^j)^\ast
    +X_a(m_1^i)h_{ij}\tau_a^\ast(m_2^j)^\ast\\
  &=  \sigma_a(m_1^i)h_{ij}X_{a}^\ast(m_2^j)^\ast
    +X_a(m_1^i)h_{ij}\sigma_{\iota(a)}(m_2^j)^\ast\\
  &=  h\paraa{\sigmahz_a(m_1),\nabla^0_{X_a^\ast}m_2}
    + h\paraa{\nabla^0_{X_a}m_1,\sigmahz_{\iota(a)}(m_2)},
\end{align*}
motivating the following definition.

\begin{definition}
  Let $\Sigma=(\A,\{X_a\}_{a\in I},\iota)$ be a $\st$-$\ast$-algebra,
  let $\nabla$ be a left $\sigmatau$-connection on the $\Sigma$-module
  $(M,\{(\hat{\sigma}_{a}, \hat{\tau}_{a})\}_{a \in I})$ and let $h$
  be a $\sigmatau$-invariant hermitian form on $M$. If
  \begin{align}\label{eq:metric.compatibility}
    X_a\paraa{h(m_1,m_2)}
    = h\paraa{\sigmah_a(m_1),\nabla_{X_{\iota(a)}}m_2}
    +h\paraa{\nabla_{X_a}m_1,\sigmah_{\iota(a)}(m_2)}
  \end{align}
  for all $m_1,m_2\in M$ and $a\in I$, then we say that \emph{$\nabla$
    is compatible with $h$} and that $\nabla$ is a \emph{metric
    connection}.
\end{definition}

\noindent
As expected, it is sufficient to check condition
\eqref{eq:metric.compatibility} on a set of generators. Let us
formulate this as follows.

\begin{proposition}\label{prop:metric.cond.on.generators}
  Let $\Sigma=(\A,\{X_a\}_{a\in I},\iota)$ be a $\st$-$\ast$-algebra,
  let $\nabla$ be a left $\sigmatau$-connection on the $\Sigma$-module
  $(M,\{(\hat{\sigma}_{a}, \hat{\tau}_{a})\}_{a \in I})$ and let $h$
  be a $\sigmatau$-invariant hermitian form on $M$.
  If $\{e_i\}_{i=1}^n$ is a set of generators of $M$ such that
  \begin{align}\label{eq:hermitian.compatible.generators}
    X_a\paraa{h(e_i,e_j)}
    = h\paraa{\sigmah_a(e_i),\nabla_{X_{\iota(a)}}e_j}
    +h\paraa{\nabla_{X_a}e_i,\sigmah_{\iota(a)}(e_j)}    
  \end{align}
  for $a\in I$ and $i,j=1,\ldots,n$, then $\nabla$ is compatible with
  $h$.
\end{proposition}

\begin{proof}
  Assume that \eqref{eq:hermitian.compatible.generators} is satisfied
  for $a\in I$ and $i,j=1,\ldots,n$, and let $m,n\in M$ be arbitrary
  elements written as $m=m^ie_i$ and $n=n^ie_i$. Then
  \begin{align*}
    h\paraa{
    &\sigmah_a(m^ie_i),\nabla_{X_{\iota(a)}}(n^je_j)}
      +h\paraa{\nabla_{X_a}(m^ie_i),\sigmah_{\iota(a)}(n^je_j)} \\
    &=\sigma_a(m^i)h\paraa{\sigmah_a(e_i),\sigma_{\iota(a)}(n^j)\nabla_{X_{\iota(a)}}e_j+X_{\iota(a)}(n^j)\tauh_{\iota(a)}(e_j)}\\
    &\qquad + h\paraa{\sigma_a(m^i)\nabla_{X_a}e_i+X_a(m^i)\tauh_a(e_i),\sigmah_{\iota(a)}(e_j)}\sigma_{\iota(a)}(n^j)^\ast \\
    &=\sigma_a(m^i)\parab{h\paraa{\sigmah_a(e_i),\nabla_{X_{\iota(a)}}e_j}\sigma_{\iota(a)}(n^j)^\ast
      +h\paraa{\sigmah_a(e_i),\tauh_{\iota(a)}(e_j)}X_{\iota(a)}(n^j)^\ast}\\
    &\qquad + \parab{\sigma_a(m^i)h\paraa{\nabla_{X_a}e_i,\sigmah_{\iota(a)}(e_j)}
      +X_a(m^i)h\paraa{\tauh_a(e_i),\sigmah_{\iota(a)}(e_j)}}\sigma_{\iota(a)}(n^j)^\ast
  \end{align*}
  and using \eqref{eq:hermitian.compatible.generators} one finds that
  \begin{align*}
    h\paraa{
    &\sigmah_a(m^ie_i),\nabla_{X_{\iota(a)}}(n^je_j)}
      +h\paraa{\nabla_{X_a}(m^ie_i),\sigmah_{\iota(a)}(n^je_j)} \\
    &=\sigma_a(m^i)X_a\paraa{h(e_i,e_j)}\sigma_{\iota(a)}(n^j)^\ast
      +\sigma_a(m^i)h\paraa{\sigmah_a(e_i),\tauh_{\iota(a)}(e_j)}X_{\iota(a)}(n^j)^\ast\\
    &\qquad + X_a(m^i)h\paraa{\tauh_a(e_i),\sigmah_{\iota(a)}(e_j)}\sigma_{\iota(a)}(n^j)^\ast.
  \end{align*}
  Now, since $h$ is $\st$-invariant and $\sigma_{\iota(a)}=\tau_a^\ast$, one obtains
  \begin{align*}
    h\paraa{
    &\sigmah_a(m^ie_i),\nabla_{X_{\iota(a)}}(n^je_j)}
      +h\paraa{\nabla_{X_a}(m^ie_i),\sigmah_{\iota(a)}(n^je_j)} \\
    &=\sigma_a(m^i)X_a\paraa{h(e_i,e_j)}\tau_a((n^j)^\ast)
      +\sigma_a(m^i)\sigma_a\paraa{h(e_i,e_j)}X_a((n^j)^\ast)\\
    &\qquad +X_a(m^i)\tau_a(h(e_i,e_j)(n^j)^\ast)\\
    &= \sigma_a(m^i)X_a\paraa{h(e_i,e_j)(n^j)^\ast} + X_a(m^i)\tau_a(h(e_i,e_j)(n^j)^\ast)\\ 
    &=X_a\paraa{m^ih(e_i,e_j)(n^j)^\ast} = X_a\paraa{h(m,n)}
  \end{align*}
  using the fact that $X_a$ is a $(\sigma_a,\tau_a)$-derivation.
\end{proof}

\noindent
Now, do metric connections exist? In the following result, we
explicitly construct metric connections on free modules.

\begin{proposition}\label{prop:metric.connection.free.module}
  Let $\Sigma=(\A,\{X_a\}_{a\in I},\iota)$ be a
  $\sigmatau$-$\ast$-algebra and let
  $(\A^n,\{\sigmah_a,\tauh_a\}_{a\in I})$ be a free $\Sigma$-module
  with basis $e_1,\ldots,e_n$. Furthermore, let $h$ be a
  $\sigmatau$-invariant hermitian form on $\A^n$ and assume that for
  each $a\in I$, there exist $h_{\sigma_a}^{ij}\in\A$ such that
  \begin{align}\label{eq:hsigma.invertible.cond}
    h_{\sigma_a}^{ij}h\paraa{e_j,\sigmah_a(e_k)}
    =\delta_k^i\mid .
  \end{align}
  Given arbitrary $\gamma_{a.ij}\in\A$, such that
  $\gamma_{a,ij}^\ast=\gamma_{\iota(a),ji}$ for $a\in I$ and
  $i,j=1,\ldots,n$,
  \begin{align}\label{eq:metric.connection.def}
    \nabla_{X_a}e_i = \parab{\tfrac{1}{2}X_a(h_{ij})+i\gamma_{a,ij}}h_{\sigma_{\iota(a)}}^{jk}e_k ,
  \end{align}
  with $h_{ij}=h(e_i,e_j)$, defines a left $\sigmatau$-connection that
  is compatible with $h$.
\end{proposition} 

\begin{proof}
  Since the module is free, any choice of $\nabla_{X_a}e_i$ defines a
  left $\sigmatau$-connection by setting (cf. Proposition~\ref{sigma
    tau conn arb})
  \begin{align*}
    \nabla_{X_a}(m^ie_i) = \sigma_a(m^i)\nabla_{X_a}e_i + X_a(m^i)\tauh_a(e_i).
  \end{align*}
  Moreover, Proposition~\ref{prop:metric.cond.on.generators} implies
  that it is sufficient to prove the metric condition
  \eqref{eq:metric.compatibility} for the basis of the module; i.e.
  \begin{align*}
    X_a(h_{ij}) = h\paraa{\sigmah_a(e_i),\nabla_{X_a^\ast}e_j}
    + h\paraa{\nabla_{X_a}e_i,\sigmah_{\iota(a)}(e_j)}
  \end{align*}
  for $a\in I$ and $i,j=1,\ldots,n$. One computes
  \begin{align*}
    h\paraa{\nabla_{X_a}e_i,\sigmah_{\iota(a)}(e_j)}
    &= \parab{\tfrac{1}{2}X_a(h_{ik})+i\gamma_{a,ik}}h_{\sigma_{\iota(a)}}^{kl}
      h\paraa{e_l,\sigmah_{\iota(a)}(e_j))}
    =\tfrac{1}{2}X_a(h_{ij}) + i\gamma_{a,ij}
  \end{align*}
  by using \eqref{eq:hsigma.invertible.cond}. Analogously, one finds
  that
  \begin{align*}
    h\paraa{\sigmah_a(e_i),&\nabla_{X_a^\ast}e_j}
    = h\paraa{\sigmah_a(e_i),\nabla_{X_{\iota(a)}}e_j}\\
    &=h\paraa{\sigmah_a(e_i),e_l}
      \parab{\paraa{\tfrac{1}{2}X_{\iota(a)}(h_{jk})+i\gamma_{\iota(a),jk}}h_{\sigma_{\iota^2(a)}}^{kl}}^\ast\\
    &= h\paraa{\sigmah_a(e_i),e_l}(h_{\sigma_{a}}^{kl})^\ast
      \parab{\tfrac{1}{2}X_{\iota(a)}^\ast(h_{kj})-i\gamma_{\iota(a),jk}^\ast}
    =\tfrac{1}{2}X_a(h_{ij}) - i\gamma_{a,ij}
  \end{align*}
  by using (the conjugate of) \eqref{eq:hsigma.invertible.cond}
  together with the assumption $\gamma_{a,ij}^\ast = \gamma_{\iota(a),ji}$. It follows that 
  \begin{align*}
    h\paraa{\sigmah_a(e_i),\nabla_{X_a^\ast}e_j}
    + h\paraa{\nabla_{X_a}e_i,\sigmah_{\iota(a)}(e_j)}
    = X_a(h_{ij}),
  \end{align*}
  showing that the $\sigmatau$-connection is compatible with $h$.
\end{proof}

\begin{remark}
  Note that if $\sigmah_a(e_i)=e_i$ for $a\in I$ and $1\leq i\leq n$,
  which is e.g. the case for the $\Sigma$-module
  $(\A^n,\{\sigmah_a^0,\tauh_a^0)\}_{a\in I}$) defined in
  Example~\ref{ex:free.module.sigma.structure}, then condition
  \eqref{eq:hsigma.invertible.cond} is equivalent to the hermitian
  form $h$ being invertible. Thus, in this case,
  Proposition~\ref{prop:metric.connection.free.module} shows that
  metric $\st$-connections exist on free modules with invertible
  hermitian forms (one can for instance always choose $\gamma_{a,ij}=0)$.
\end{remark}

\noindent
In the next result we show that one can construct metric connections
on projective $\Sigma$-modules by using orthogonal projections.

\begin{proposition}\label{prop:metric.connection.projective.module}
  Let $\nablat$ be a left $\sigmatau$-connection on the free
  $\Sigma$-module
  $(\A^n, \{(\hat{\sigma}_{a}, \hat{\tau}_{a})\}_{a \in I})$ that is
  compatible with the $\sigmatau$-invariant hermitian form $h$, and let
  $p:\A^n\to\A^n$ be a projection that is orthogonal with respect to
  $h$. If $[\sigmah_a,p]=[\tauh_a,p]=0$ for all $a\in I$ then
  $h|_{p(\A^n)}$ is a $\sigmatau$-invariant hermitian form on
  $(p(\A^n), \{(p\circ\hat{\sigma}_{a}, p\circ\hat{\tau}_{a})\}_{a \in
    I})$, and $\nabla=p\circ\nablat$ is a left $\sigmatau$-connection
  on
  $(p(\A^n), \{(p\circ\hat{\sigma}_{a}, p\circ\hat{\tau}_{a})\}_{a \in
    I})$ that is compatible with $h|_{p(\A^n)}$. 
\end{proposition}

\begin{proof}
  First of all, it follows from Proposition~\ref{con on proj mod} that
  $\nabla$ is a $\sigmatau$-connection on
  $(p(\A^n), \{(p\circ\hat{\sigma}_{a}, p\circ\hat{\tau}_{a})\}_{a\in
    I})$. Moreover, since $[\sigmah_a,p]=[\tauh_a,p]=0$ it follows
  that
  \begin{align*}
    h\paraa{(p\circ\sigmah_a)(m_1),(p\circ\tauh_{\iota(a)})(m_2)}
    &= h\paraa{(\sigmah_a\circ p)(m_1),(\tauh_{\iota(a)}\circ p)(m_2)}\\
    &= h\paraa{\sigmah_a(m_1),\tauh_{\iota(a)}(m_2)}
      =\sigma_a\paraa{h(m_1,m_2)}
  \end{align*}
  for $m_1,m_2\in p(\A^n)$, since $h$ is $\sigmatau$-invariant. Hence, $h|_{p(\A^n)}$ is
  $\sigmatau$-invariant on
  $(p(\A^n), \{(p\circ\hat{\sigma}_{a}, p\circ\hat{\tau}_{a})\}_{a \in
    I})$. Next, let us show that $\nabla$ is compatible with
  $h|_{p(\A^n)}$. Let $m_1,m_2\in p(\A^n)$ and compute
  \begin{align*}
    h\paraa{(&p\circ\sigmah_a)(m_1),\nabla_{X_a^\ast}m_2}
               +h\paraa{\nabla_{X_a}m_1,(p\circ\sigmah_{\iota(a)})(m_2)}\\
             &= h\paraa{(p\circ\sigmah_a)(m_1),p(\nablat_{X_a^\ast}m_2)}
               +h\paraa{p(\nablat_{X_a}m_1),(p\circ\sigmah_{\iota(a)})(m_2)}\\
             &=h\paraa{(p^2\circ\sigmah_a)(m_1),\nablat_{X_a^\ast}m_2}
               +h\paraa{\nablat_{X_a}m_1,(p^2\circ\sigmah_{\iota(a)})(m_2)}\\
             &=h\paraa{\sigmah_a(m_1),\nablat_{X_a^\ast}m_2}
               +h\paraa{\nablat_{X_a}m_1,\sigmah_{\iota(a)}(m_2)}
               = X_a\paraa{h(m_1,m_2)}
  \end{align*}
  using that $\nablat$ is compatible with $h$ and
  $[\sigmah_a,p]=0$. We conclude that $\nabla$ is compatible with
  $h|_{p(\A^n)}$.
\end{proof}

\noindent
Combining this result with Proposition~\ref{proj mod iso}, which in
particular gives a projection $p$ such that
$[\sigmah_a,p]=[\tauh_a,p]=0$, one can construct metric
$\st$-connections on projective modules from metric $\st$-connections
on free modules as soon as the projection is orthogonal.

\section{Torsion and curvature}

\noindent
In differential geometry, the fact that the set of derivations (or,
correspondingly, the vector fields) is a Lie algebra, with respect to
composition of derivations, is used to define both torsion and
curvature
\begin{align*}
  &T(X,Y) =\nabla_XY-\nabla_YX-[X,Y]\\ 
  &R(X,Y)Z = \nabla_X\nabla_YZ-\nabla_Y\nabla_XZ-\nabla_{[X,Y]}Z.
\end{align*}
So far, the $\st$-derivations of a $\st$-algebra are not assumed to
satisfy any kind of commutation relation. However, in many interesting
cases the $\st$-derivations satisfy twisted commutation relations;
e.g. in case of the quantum 3-sphere one finds that
\begin{align*}
  &\Xm\Xp - q^2\Xp\Xm = \Xz\\
  q^2&\Xz\Xm-q^{-2}\Xm\Xz=(1+q^2)\Xm\\
  q^2&\Xp\Xz-q^{-2}\Xz\Xp=(1+q^2)\Xp.
\end{align*}
In the general case, we introduce the following notion of a $\st$-Lie algebra.

\begin{definition}
  Let $\Sigma = (\A,\{X_a\}_{a\in I})$ be a $\st$-algebra. Given a map
  \begin{equation*}
    R:\TSigma\otimesC\TSigma\to\TSigma\otimesC\TSigma 
  \end{equation*}
  we say that $(\TSigma,R)$ is a \emph{$\st$-Lie algebra} if
  \begin{enumerate}
  \item $R^2=\id_{\TSigma\otimesC\TSigma}$,
  \item $m\paraa{X_a\otimesC X_b - R(X_a\otimesC X_b)} \in\TSigma$,
  \end{enumerate}
  for all $a,b\in I$, where
  $m:\TSigma\otimesC\TSigma\to\operatorname{End_\complex}(\A)$ denotes the
  composition map given by $m(X\otimesC Y)=X\circ Y$ for $X,Y\in\TSigma$.
\end{definition}

\noindent
This definition is related to the definition of a quantum/braided Lie-algebra
(see e.g. \cite{b:quantum.lie.algebras,m:quantum.braided.lie}); however, the ``$R$-matrix'' is not
necessarily assumed to satisfy a Yang-Baxter equation nor any braided version of the Jacobi identity.

Given a $\st$-Lie algebra $(\TSigma,R)$ one introduces components
of $R$ via
\begin{align*}
  R(X_a\otimesC X_b) = R_{ab}^{pq}X_p\otimesC X_q
\end{align*}
as well as structure constants $C_{ab}^p$
\begin{align*}
  m\paraa{X_a\otimesC X_b - R(X_a\otimesC X_b)} = C_{ab}^pX_p
\end{align*}
and the bilinear bracket
$\comR{\cdot,\cdot}:\TSigma\otimesC\TSigma\to\TSigma$
\begin{align*}
  \comR{X,Y} = m\paraa{X\otimesC Y-R(X\otimesC Y)}
\end{align*}
giving
\begin{align}\label{eq:comR.def}
  \comR{X_a,X_b} = X_a\circ X_b - R_{ab}^{pq}X_p\circ X_q = C_{ab}^pX_p.
\end{align}
The assumption that $R^2=\id_{\TSigma\otimesC\TSigma}$ can the be
expressed as
\begin{align*}
  \paraa{R_{ab}^{pq}R_{pq}^{rs}-\delta_a^r\delta_b^s}X_r\otimesC X_s= 0 
\end{align*}
and it follows that the bracket exhibits an $R$-antisymmetry
\begin{align*}
  R_{ab}^{pq}\comR{X_p,X_q}
  &= R_{ab}^{pq}X_p\circ X_q-R_{ab}^{pq}R_{pq}^{rs}X_r\circ X_s\\
  &=R_{ab}^{pq}X_p\circ X_q - X_a\circ X_b = -\comR{X_a,X_b},
\end{align*}
which may also be expressed in terms of the structure constants as
\begin{align*}
  R_{ab}^{pq}C_{pq}^rX_r=-C_{ab}^rX_r.
\end{align*}
Note that in the following we shall in most cases simply say that
$\TSigma$ is a $\sigmatau$-Lie algebra and tacitly assume a choice of $R$.
Let us now proceed to define curvature and torsion.

\begin{definition}\label{def:curvature}
  Let $\Sigma=\AstX$ be a $\sigmatau$-algebra, such that $(\TSigma,R)$ is a
  $\sigmatau$-Lie algebra, and let
  $(M,\{\sigmah_a,\tauh_a\}_{a\in I})$ be a $\Sigma$-module. Given a
  $\sigmatau$-connection $\nabla$, the curvature of $\nabla$ is defined as
  \begin{align}\label{eq:def.curvature}
    \Curv(X,Y)n = m_{\nabla}\paraa{X\otimesC Y-R(X\otimesC Y)}n-\nabla_{\comR{X,Y}}n
  \end{align}
  for $n\in M$ and $X,Y\in\TSigma$, where
  $m_{\nabla}(X\otimesC Y)=\nabla_X\circ\nabla_Y$.
\end{definition}

\noindent
For a left $\sigmatau$-connection $\nabla$ on a $\Sigma$-module $M$ and
for a set of generators $e_1,\ldots,e_n$, one writes
$\nabla_{X_a}e_i=\Gamma_{ai}^je_j$ and finds that
\begin{align*}
  \Curv(X_a,X_b)e_i
  = \paraa{\sigma_a(\Gamma_{bi}^j)&\Gamma_{aj}^k-R_{ab}^{pq}\sigma_p(\Gamma_{qi}^j)\Gamma_{pj}^k
  -C_{ab}^{c}\Gamma_{ci}^k}e_k\\
  &+X_a(\Gamma_{bi}^k)\tauh_a(e_k)-R_{ab}^{pq}X_p(\Gamma_{qi}^k)\tauh_p(e_k).
\end{align*}

\noindent
In order to define torsion, one needs a way to identify
$\st$-derivations with elements of the module. Such a map is sometimes
called an anchor map in differential geometry. Note that in the
classical case, where the module is the module of vector fields, the
anchor map is simply the isomorphism between derivations and vector
fields. In general, we make the following definition, in analogy with
the framework of real calculi developed in
\cite{aw:cgb.sphere,aw:curvature.three.sphere,atn:minimal.embeddings}.

\begin{definition}
  Let $\Sigma=\AstX$ be a $\sigmatau$-algebra and let
  $(M,\{\sigmah_a,\tauh_a\}_{a\in I})$ be a $\Sigma$-module. An
  \emph{anchor map} is a $\complex$-linear map $\varphi:\TSigma\to M$.
\end{definition}

\noindent
Given the choice of an anchor map, it is straightforward to define
torsion.

\begin{definition}\label{def:torsion}
  Let $\Sigma=\AstX$ be a $\sigmatau$-algebra, such that $(\TSigma,R)$ is a
  $\sigmatau$-Lie algebra, and let
  $(M,\{\sigmah_a,\tauh_a\}_{a\in I})$ be a $\Sigma$-module. Given a
  $\sigmatau$-connection $\nabla$ and an anchor map
  $\varphi:\TSigma\to M$, the \emph{torsion of $\nabla$ with respect to $\varphi$} is defined as
  \begin{align}\label{eq:def.torsion}
    T(X,Y) = m_{\varphi}\paraa{X\otimesC Y-R(X\otimesC Y)}-\varphi\paraa{\comR{X,Y}}
  \end{align}
  for $X,Y\in\TSigma$ with
  $m_{\varphi}(X\otimesC Y)=\nabla_{X}\varphi(Y)$. The
  connection is called \emph{torsion free} (with respect to $\varphi$)
  if $T(X,Y)=0$ for all $X,Y\in\TSigma$.
\end{definition}

\noindent Let us illustrate the above concepts with a simple
example.
\begin{example}
Let $\A=\complex[x,y]$ be the (commutative)
polynomial algebra in two variables. For $q\in\reals$ (with $q\neq 1$)
and $f\in\complex[x,y]$, let us introduce the Jackson derivatives
\begin{align*}
  X_1(f)(x,y) = \frac{f(qx,y)-f(x,y)}{(q-1)x}\qand
  X_2(f)(x,y) = \frac{f(x,qy)-f(x,y)}{(q-1)y}
\end{align*}
giving
\begin{align*}
  X_1(x^ny^m) = [n]_qx^{n-1}y^m\qand
  X_2(x^ny^m) = [m]_qx^ny^{m-1}
\end{align*}
where $[n]_q=(q^n-1)/(q-1)=1+q+q^2+\cdots+q^{n-1}$. Furthermore, for
$f,g\in\complex[x,y]$ it is easy to check that
\begin{align*}
  & X_1(fg) = \sigma_1(f)X_1(g) + X_1(f)g\\
  & X_2(fg) = \sigma_2(f)X_2(g) + X_2(f)g
\end{align*}
where $\sigma_1(f)(x,y)=f(qx,y)$ and $\sigma_2(f)(x,y)=f(x,qy)$,
as well as $[X_1,X_2] = 0$. Hence, 
\begin{align*}
  \Sigma = (\complex[x,y],\{X_1,X_2\})
\end{align*}
is a $\st$-algebra (with $\tau_1=\tau_2=\id_\A$). Moreover, it is
clear that $(\TSigma,R)$ is a $\st$-Lie algebra with
\begin{align*}
  &R(X_a\otimesC X_b) = X_b\otimesC X_a\\
  &m\paraa{X_a\otimesC X_b-R(X_a\otimesC X_b)} = 0.
\end{align*}
for $a,b\in\{1,2\}$. Next, let $\{e_1,e_2\}$ be a basis of the free module $M=\A^2$. By
setting
\begin{align*}
  \sigmah_1(m^ae_a) = \sigma_1(m^a)e_a\qand
  \sigmah_2(m^ae_a) = \sigma_2(m^a)e_a
\end{align*}
it is easy to see that $(M,\{(\sigmah_1,\id_M),(\sigmah_2,\id_M)\})$
is a free $\Sigma$-module. A $\st$-connection on $M$ is given by a
choice of $\Gamma_{ab}^c\in\A$ via
\begin{align*}
  \nabla_{X_a}e_b = \Gamma_{ab}^ce_c
\end{align*}
and extending it to $M$ as a $\st$-connection (cf. Proposition~\ref{sigma tau conn arb})
\begin{align*}
  \nabla_{X_a}(m^be_b) = \sigma_a(m^b)\nabla_{X_a}e_b + X_a(m^b)e_b.
\end{align*}
By defining an anchor map $\varphi:\TSigma\to M$ as
$\varphi(X_a)=e_a$, torsion free connections (with respect to $\varphi$) satisfy
\begin{align*}
  \nabla_{X_a}e_b - \nabla_{X_b}e_a = 0
\end{align*}
which is equivalent to $\Gamma_{ab}^c=\Gamma_{ba}^c$ for
$a,b,c\in\{1,2\}$. Moreover, since 
\begin{align*}
  [X_1,X_2] = [X_1,\sigma_2]=[X_2,\sigma_1]=[\sigma_1,\sigma_2] = 0,
\end{align*}
the curvature exhibits a twisted linearity property
\begin{align*}
  \Curv&(X_1,X_2)(fm)
  = \nabla_{X_1}\nabla_{X_2}(fm)-\nabla_{X_2}\nabla_{X_1}(fm)\\
       &= \nabla_{X_1}\paraa{\sigma_2(f)\nabla_{X_2}m+X_2(f)m}
         -\nabla_{X_2}\paraa{\sigma_1(f)\nabla_{X_1}m+X_1(f)m}\\
       &= \sigma_1(\sigma_2(f))\nabla_{X_1}\nabla_{X_2}m+X_1(\sigma_2(f))\nabla_{X_2}m
         +\sigma_1(X_2(f))\nabla_{X_1}m+X_1(X_2(f))m\\
       &-\sigma_2(\sigma_1(f))\nabla_{X_2}\nabla_{X_1}m-X_2(\sigma_1(f))\nabla_{X_1}m
         -\sigma_2(X_1(f))\nabla_{X_2}m - X_2(X_1(f))m\\
       &=\paraa{\sigma_1\circ\sigma_2}(f)\Curv(X_1,X_2)m.
\end{align*}
For instance, choosing
\begin{alignat*}{2}
  &\nabla_{X_1}e_1 = y^me_2 &\qquad &\nabla_{X_2}e_2 = x^ne_1\\
  &\nabla_{X_1}e_2 = 0 &\qquad &\nabla_{X_2}e_1 = 0
\end{alignat*}
one obtains a torsion free connection, and the curvature can be
computed as
\begin{align*}
  \Curv(X_1,X_2)e_1
  &= \nabla_{X_1}\nabla_{X_2}e_1-\nabla_{X_2}\nabla_{X_1}e_1 \\
  &= -\nabla_{X_2}\paraa{y^me_2}
    = -\sigma_2(y^m)\nabla_{X_2}e_2 - X_2(y^m)e_2\\
  &= -q^mx^ny^me_1 - [m]_qy^{m-1}e_2\\
  \Curv(X_1,X_2)e_2
  &= \nabla_{X_1}\nabla_{X_2}e_2 -\nabla_{X_2}\nabla_{X_1}e_2\\
  &= \nabla_{X_1}\paraa{x^ne_1} = \sigma_1(x^n)\nabla_{X_1}e_1+X_1(x^n)e_1\\
  &= q^nx^ny^me_2 + [n]_qx^{n-1}e_1
\end{align*}
which can be extended to all of $M$ by using the twisted linearity
property of the curvature
\begin{align*}
  \Curv(X_1,X_2)(fm)=\paraa{\sigma_1\circ\sigma_2}(f)\Curv(X_1,X_2)m,
\end{align*}
giving, for instance,
\begin{align*}
  \Curv(X_1,X_2)(xye_1) &= \sigma_1(\sigma_2(xy))\Curv(X_1,X_2)e_1 = (q^2xy)\Curv(X_1,X_2)e_1\\
  &= -q^{m+2}x^{n+1}y^{m+1}e_1-q^2[m]_qxy^me_2.
\end{align*}
\end{example}

\noindent
In the general, both the structure of the $\st$-Lie algebra, as
well as the choice of an anchor map, affect the existence of torsion
free connections.  However, let us prove a result which gives a way of
constructing torsion free connections on free modules.

\begin{proposition}\label{prop:torsion.free.sufficient}
  Let $\Sigma=\AstX$ be a $\sigmatau$-algebra and assume that
  $(\TSigma,R)$ is a $\sigmatau$-Lie algebra. Moreover, let
  $(M,\{\sigmah_a,\tauh_a\}_{a\in I})$ be a left $\Sigma$-module and
  let $\nabla$ be a left $\sigmatau$-connection. If there exist
  $\gamma_{ab}^c\in\A$ for $a,b,c\in I$ such that
  $R_{ab}^{pq}\gamma_{pq}^c=\gamma_{ab}^c$
  and
  \begin{align}
    \nabla_{X_a}\varphi(X_b) = \paraa{\tfrac{1}{2}C_{ab}^c + \gamma_{ab}^c}\varphi(X_c)
  \end{align}
  for $a,b\in I$, then $\nabla$ is torsion free with respect to $\varphi$.
\end{proposition}

\begin{proof}
  One easily checks that \eqref{eq:def.torsion} holds:
  \begin{align*}
    \nabla_{X_a}\varphi(X_b)
    &-R^{pq}_{ab}\nabla_{X_p}\varphi(X_q)
    -\varphi\paraa{\comR{X_a,X_b}}\\
    &=\paraa{\tfrac{1}{2}C_{ab}^c+\gamma_{ab}^c}\varphi(X_c)
      -R_{ab}^{pq}\paraa{\tfrac{1}{2}C_{pq}^c+\gamma_{pq}^c}\varphi(X_c)-C_{ab}^c\varphi(X_c)\\
    &= \paraa{\tfrac{1}{2}C_{ab}^c+\tfrac{1}{2}C_{ab}^c+\gamma_{ab}^c-\gamma_{ab}^c-C_{ab}^c}\varphi(X_c)=0
  \end{align*}
  by using that
  $R_{ab}^{pq}C_{ab}^c\varphi(X_c)=-C_{ab}^c\varphi(X_c)$ and
  $R_{ab}^{pq}\gamma_{pq}^c=\gamma_{ab}^c$.
\end{proof}

\noindent
In the particular case when the anchor map $\varphi$ is such that
$\{\varphi(X_a)\}_{a\in I}$ is a basis of the module $M$,
Proposition~\ref{prop:torsion.free.sufficient} implies that one may
define a torsion free connection by setting
\begin{align*}
  \nabla_{X_a}\varphi(X_b) = \paraa{\tfrac{1}{2}C_{ab}^c + \gamma_{ab}^c}\varphi(X_c)
\end{align*}
and extending $\nabla$ to $M$ as a $\sigmatau$-connection.  Moreover, note
that one can easily find $\gamma_{ab}^c$ such that
$R_{ab}^{pq}\gamma_{pq}^c=\gamma_{ab}^c$; namely, for any
$\tilde{\gamma}_{ab}^c\in\A$ one readily checks that
\begin{align*}
  \gamma_{ab}^c = \tilde{\gamma}_{ab}^c+R_{ab}^{pq}\tilde{\gamma}_{pq}^c
\end{align*}
satisfies $R_{ab}^{pq}\gamma_{pq}^c=\gamma_{ab}^c$ due to
$R_{ab}^{pq}R_{pq}^{cd}=\delta_a^c\delta_b^d$.

Finally, let us define the concept of a Levi-Civita $\st$-connection.

\begin{definition}
  Let $\Sigma=(\A,\{X_a\}_{a\in I},\iota)$ be a
  $\sigmatau$-$\ast$-algebra, such that $(\TSigma,R)$ is a
  $\sigmatau$-Lie algebra, and let
  $(M,\{\sigmah_a,\tauh_a\}_{a\in I})$ be a $\Sigma$-module. Moreover,
  let $\varphi:\TSigma\to M$ be an anchor map and let $h$ be a
  $\st$-invariant hermitian form on $M$. If $\nabla$ is a torsion free
  $\st$-connection with respect to $\varphi$, such that $\nabla$ is
  compatible with $h$, then $\nabla$ is called a \emph{Levi-Civita
    $\st$-connection}.
\end{definition}

\noindent
In general, for an arbitrary choice of $\st$-algebra, $\Sigma$-module,
hermitian form and anchor map, the existence of a Levi-Civita
$\st$-connection cannot be guaranteed, but let us illustrate in the
next section that such connections can be constructed by studying an
example over the matrix algebra.

\section{Levi-Civita $\st$-connections over matrix algebras}\label{sec:lc.matrix}

\noindent
In this section we will construct $\st$-algebras over matrix algebras
as an illustration of the concepts that have been introduced in the
previous sections. Let $\A=\MatNC$ denote the algebra of
complex $(N\times N)$-matrices and let $\{U_a\}_{a=1}^n$ denote a set
of invertible matrices such that $[U_a,U_b]=0$ for
$a,b=1,\ldots,n$. Defining $\sigma_a(A)=U_aAU_a^{-1}$  and $X_a:\MatNC\to\MatNC$ as
\begin{align}
  X_a(A) = A-\sigma_a(A) = A -U_aAU_a^{-1} 
\end{align}
for $A\in\MatNC$ and $a=1,\dots,n$, one easily checks that $X_a$ is a
$(\sigma_a,\idA)$-derivation (in fact, one can check that $X_a$ is a
$(\idA,\sigma_a)$-derivation as well). With this notation, we
introduce the $\st$-algebra
\begin{align*}
  \Sigma = (\MatNC,\{X_a\}_{a=1}^n)
\end{align*}
where $X_a$ satisfies
\begin{align*}
  X_a(AB) = \sigma_a(A)X_a(B) + X_a(A)B =U_aAU_a^{-1}X_a(B)+X_a(A)B
\end{align*}
for $A,B\in\MatNC$. The tangent space $\TSigma$ is generated by $\{X_a\}_{a=1}^n$, and since
$[U_a,U_b]=0$ it follows that $[X_a,X_b]=0$ for $a,b=1,\ldots,n$. Hence,
$(\TSigma,R)$ is a $\st$-Lie algebra with
\begin{equation*}
  R(X\otimesC Y)=Y\otimesC X
\end{equation*}
and
\begin{equation*}
  m\paraa{X\otimesC Y-R(X\otimesC Y)} = 0
\end{equation*}
for $X,Y\in\TSigma$. In the following, we will assume that
$\{X_a\}_{a=1}^n$ is a basis of $\TSigma$, and one finds that
\begin{equation*}
  R_{ab}^{pq} =\delta_a^q\delta_b^p\qand
  C_{ab}^p = 0
\end{equation*}
for $a,b,p,q=1,\ldots,n$.

For the $\st$-algebra constructed above, we shall
consider connections on two types of modules: the algebra $\MatNC$
itself, i.e. a free module of rank one, and $\complex^N$, which is a
projective module (that is not free). To be more specific, let us
recall how one may realize $\complex^N$ in terms of a projector. We
consider elements of $\complex^N$ as column vectors, and matrices act
in the standard way on such vectors from the left, giving a left
module structure on $\complex^N$. Now, for any $v_0\in\complex^N$ such
that $|v_0|=1$, one defines the matrix $p=v_0v_0^\dagger$, where
$v_0^\dagger$ denotes the hermitian conjugate of $v_0$. It is easy to
see that $p$ is a projection, i.e.
\begin{align*}
  &p^2 = v_0v_0^\dagger v_0v_0^\dagger = v_0|v_0|^2v_0^\dagger = v_0v_0^\dagger = p\\
  &p^\dagger = (v_0v_0^\dagger)^\dagger = v_0v_0^\dagger = p.
\end{align*}
since $v_0^\dagger v_0=|v_0|^2=1$. Hence, $M=\MatNC p$ is a (left)
projective module. This module is indeed isomorphic to $\complex^N$,
and the isomorphism can be realized by the module homomorphism $\phi:\complex^N\to\MatNC p$
given by
\begin{align}\label{eq:phi.isomorphism}
  \phi(v) = vv_0^\dagger
\end{align}
for $v\in\complex^N$, with inverse $\phi^{-1}(A)=Av_0$ for
$A\in\MatNC p$. By a slight abuse of notation, we will in the
following also denote by $p$ the map $p:\MatNC\to\MatNC$ given by
$p(A)=Ap$ for $A\in\MatNC$.

\subsection{Torsion free $\st$-connections on projective $\Sigma$-modules}

\noindent
First of all, note that $(\MatNC,\{\sigma_a,\idA\}_{a=1}^n)$ is a free
$\Sigma$-module. Moreover, for an arbitrary projection $p\in\MatNC$, consider
the projective module $\Mp=\MatNC p$; the important cases will be when
$p=\mid$ (the $N\times N$ identity matrix) and $p=v_0v_0^\dagger$ for
$v_0\in\complex^N$ such that $|v_0|=1$.  According to
Proposition~\ref{prop:TM}, one can equip $\Mp$ with the structure of a
$\Sigma$-module by defining
\begin{align*}
  \sigmah_a = p\circ\sigma_a\qand
  \tauh_a = p\circ\id_{\A},
\end{align*}
giving
\begin{align*}
  \sigmah_a(A) = \sigma_a(A)p=U_aAU_a^{-1}p\qquad
  \tauh_a(A) = Ap = A
\end{align*}
for $A\in \Mp$. Since $(\MatNC,\{\sigma_a,\idA\}_{a=1}^n)$ is a free
$\Sigma$-module of rank one with a basis given by the identity matrix
$\mid$, Proposition~\ref{sigma tau conn arb} implies that a
$\st$-connection may be defined by specifying, for $a=1,\ldots,n$
\begin{align*}
  \nablat_{X_a}\mid = \Gammat_a
\end{align*}
for arbitrary $\Gammat_a\in\MatNC$, giving
\begin{align*}
  \nablat_{X_a}A = \sigma_a(A)\nabla_{X_a}\mid + X_a(A)\mid = A-U_aAU_a^{-1}(\mid-\Gammat_a)
  =A-U_aAU_a^{-1}\Gamma_a
\end{align*}
for $A\in\MatNC$, where $\Gamma_a = \mid-\Gammat_a$. Moreover,
Proposition~\ref{con on proj mod} implies that $p\circ\nablat$ is a
$\st$-connection on
$(\Mp,\{\sigmah_a,\id_{\Mp}\}_{a=1}^n)$; i.e.
\begin{align}\label{eq:st.connection.proj.matrix}
  \nabla_{X_a}A = Ap-U_aAU_a^{-1}\Gamma_ap = A-U_aAU_a^{-1}\Gamma_ap
\end{align}
for $A\in\Mp$. In the case when $p=v_0v_0^\dagger$, let us use the
isomorphism $\phi$ (as defined in \eqref{eq:phi.isomorphism}) to write
down the connection explicitly acting an element $v\in\complex^N$
\begin{align*}
  \nabla_{X_a}v
  &= \phi^{-1}\paraa{\nabla_{X_a}\phi(v)}
    =\paraa{\nabla_{X_a}vv_0^\dagger}v_0\\
  &= vv_0^\dagger v_0-U_avv_0^\dagger U_a^{-1}\Gamma_av_0v_0^\dagger v_0
    = v- U_av(v_0^\dagger U_a^{-1}\Gamma_av_0)\\
  &= (\mid-\gamma_aU_a)v
\end{align*}
where $\gamma_a=v_0^\dagger U_a^{-1}\Gamma_av_0\in\complex$ for
$a=1,\ldots,n$.  Given the previously introduced $\st$-Lie algebra
structure on $\TSigma$ one easily computes the curvature of $\nabla$.

\begin{proposition}\label{prop:curvature.mat}
  The curvature of the $(\sigma, \tau)$-connection
  \begin{align*}
    \nabla_{X_a}A = A-U_aAU_a^{-1}\Gamma_ap
  \end{align*}
  on $(\Mp,\{\sigmah_a,\idA\})$ is given by
  \begin{equation}\label{eq:curvature.proj.connection}
    \Curv(X_{a}, X_{b})A =  U_{a}U_{b}A[U_{b}^{-1}\Gamma_{b}p, U_{a}^{-1}\Gamma_{a}p].	
  \end{equation}
  for $A\in\Mp$. Moreover, the curvature satisfies
  \begin{align*}
    \Curv(X_a,X_b)(BA) = \sigma_a\paraa{\sigma_b(B)}\Curv(X_a,X_b)A
  \end{align*}
  for all $A\in\Mp$ and $B\in\MatNC$, and if $p=v_0v_0^\dagger$ for
  $v_0\in\complex^N$ such that $|v_0|=1$, then $\Curv(X_a,X_b)A = 0$
  for $A\in\Mp$ and $a,b=1,\ldots,n$.
\end{proposition}

\begin{proof}
  Following Definition~\ref{def:curvature} and using the
  $\st$-Lie algebra structure on $\TSigma$, one computes
  \begin{align*}
    \Curv(X_{a}, X_{b})A
    &= \nabla_{X_{a}}\nabla_{X_{b}}A - \nabla_{X_{b}}\nabla_{X_{a}}A\\
    &= \nabla_{X_{a}}(A - U_{b}AU_{b}^{-1}\Gamma_{b}p) - \nabla_{X_{b}}(A - U_{a}AU_{a}^{-1}\Gamma_{a}p)\\
    &= \nabla_{X_{a}}A - \nabla_{X_{a}}(U_{b}AU_{b}^{-1}\Gamma_{b}p) - \nabla_{X_{b}}(A) + \nabla_{X_{b}}(U_{a}AU_{a}^{-1}\Gamma_{a}p)\\
    &= A - U_{a}AU_{a}^{-1}\Gamma_{a}p - U_{b}AU_{b}^{-1}\Gamma_{b}p + U_{a}U_{b}AU_{b}^{-1}\Gamma_{b}pU_{a}^{-1}\Gamma_{a}p\\
    &\qquad- A + U_{b}AU_{b}^{-1}\Gamma_{b}p + U_{a}AU_{a}^{-1}\Gamma_{a}p - U_{b}U_{a}AU_{a}^{-1}\Gamma_{a}pU_{b}^{-1}\Gamma_{b}p\\
    &= U_{a}U_{b}A\left(U_{b}^{-1}\Gamma_{b}pU_{a}^{-1}\Gamma_{a}p - U_{a}^{-1}\Gamma_{a}pU_{b}^{-1}\Gamma_{b}p\right)\\
    &= U_{a}U_{b}A[U_{b}^{-1}\Gamma_{b}p, U_{a}^{-1}\Gamma_{a}p],
  \end{align*}
  using the fact that $[U_a,U_b]=0$.  Moreover, one finds that
  \begin{align*}
    \Curv&(X_a,X_b)(BA)
    = U_aU_bBA[U_{b}^{-1}\Gamma_{b}p, U_{a}^{-1}\Gamma_{a}p]\\
         &= (U_aU_bBU_b^{-1}U_a^{-1})U_aU_bA[U_{b}^{-1}\Gamma_{b}p, U_{a}^{-1}\Gamma_{a}p]
           = \sigma_a\paraa{\sigma_b(B)}\Curv(X_a,X_b)A
  \end{align*}
  for $B\in\MatNC$ and $A\in\Mp$. Assuming that $p=v_0v_0^\dagger$ one
  computes, for arbitrary $A_1,A_2\in\MatNC$
  \begin{align*}
    p[A_1p, A_2p]
    &=pA_1pA_2p-pA_2pA_1p\\
    &= v_0\paraa{(v_0^\dagger A_1v_0)(v_0^\dagger A_2 v_0)-(v_0^\dagger A_2v_0)(v_0^\dagger A_1 v_0)}v_0^\dagger=0
  \end{align*}
  since $v_0^\dagger A_iv_0\in\complex$ for $i=1,2$. Applying this to the curvature, one obtains
  \begin{align*}
    \Curv(X_a,X_b)A = U_aU_bA[U_{b}^{-1}\Gamma_{b}p, U_{a}^{-1}\Gamma_{a}p]
    =U_aU_bA\paraa{p[U_{b}^{-1}\Gamma_{b}p, U_{a}^{-1}\Gamma_{a}p]} = 0
  \end{align*}
  using that $A=Ap$ for $A\in\Mp$. 
\end{proof}

\noindent
Let us now construct a class of torsion free $\st$-connections on
$(\Mp,\{\sigmah_a,\idA\}_{a=1}^n)$. To this end, let
$\varphi:\MatNC\to\Mp$ be an arbitrary anchor map, and write
$\varphi(X_a)=E_ap\in\Mp$ for $E_a\in\MatNC$. With respect to the
$\st$-Lie algebra structure on $\TSigma$ given above, the
$\st$-connection defined in \eqref{eq:st.connection.proj.matrix} is
torsion free if
\begin{align*}
  0 &= \nabla_{X_a}\varphi(X_b)-\nabla_{X_b}\varphi(X_a)
  = E_bp-E_ap-U_aE_bpU_a^{-1}\Gamma_ap+U_bE_apU_b^{-1}\Gamma_bp\\
  &=\paraa{E_b-E_a-U_aE_bpU_a^{-1}\Gamma_a+U_bE_apU_b^{-1}\Gamma_b}p,
\end{align*}
for $a,b=1,\ldots,n$.  These equations may be solved in many different
ways, but let us for definiteness present a particular solution. To
this end, assume that $\{E_a\}_{a=1}^n$ is chosen such that
$[E_a,E_b]=[E_a,U_b]=0$ for $a,b=1,\ldots,n$ (one can for instance
choose $E_a=U_a$), and
set $\Gamma_a=\mid-E_a$, giving
\begin{align}
  \nabla_{X_a}A = A-U_aAU_a^{-1}(\mid-E_a)p.
\end{align}
If $p=\mid$ then one readily checks that this connection is indeed
torsion free with respect to the anchor map $\varphi$:
\begin{align*}
  \nabla_{X_a}\varphi(X_b)-\nabla_{X_b}\varphi(X_a)
  &= E_b-E_a-U_aE_bU_a^{-1}(\mid-E_a)+U_bE_aU_b^{-1}(\mid-E_b)\\
  &= E_b-E_a-E_b(\mid-E_a)+E_a(\mid-E_b)=0
\end{align*}
since $[E_a,U_b]=[E_a,E_b]=0$ for $a,b=1,\ldots,n$. Note that it
follows from Proposition~\ref{prop:curvature.mat} that this connection
has zero curvature.

In the case when $p=v_0v_0^\dagger$ let us assume that $v_0$ is a
common eigenvector of the commuting matrices $\{E_a,U_a\}_{a=1}^n$,
i.e
\begin{align*}
  E_av_0=\lambda_av_0\qand U_av_0=\mu_{a}v_0
\end{align*}
for $\lambda_a,\mu_a\in\complex$ and $a=1,\ldots,n$. It follows that
\begin{align*}
  U_aE_bpU_a^{-1}(\mid-E_a)p
  &=U_aE_bv_0v_0^\dagger U_a^{-1}(\mid-E_a)v_0v_0^\dagger
    =\mu_a\lambda_bv_0v_0^\dagger \mu_a^{-1}(1-\lambda_a)v_0v_0^\dagger\\
  &=\lambda_b(1-\lambda_a)(v_0v_0^\dagger)^2 = \lambda_b(1-\lambda_a)p
\end{align*}
giving
\begin{align*}
  \nabla_{X_a}\varphi(X_b)-\nabla_{X_b}\varphi(X_a) &=
  \paraa{E_b-E_a-U_aE_bpU_a^{-1}\Gamma_a+U_bE_apU_b^{-1}\Gamma_b}p\\
  &=\lambda_bp - \lambda_ap - \lambda_b(1-\lambda_a)p + \lambda_a(1-\lambda_b)p = 0.
\end{align*}
Hence, the $\st$-connection
\begin{align*}
  \nabla_{X_a}v = (\mid-\gamma_aU_a)v = \paraa{\mid-\mu_a^{-1}(1-\lambda_a)U_a}v
\end{align*}
is torsion free with respect to $\varphi(X_a)=\lambda_av_0$.

\subsection{$\st$-$\ast$-algebras}

\noindent
In this section, we shall assume that the matrices $\{U_a\}_{a=1}^n$
are unitary. The matrix algebra $\MatNC$ is a $\ast$-algebra with
respect to the hermitian transpose, and one finds that
\begin{align*}
  \sigma_a^\ast(A) = \sigma_a(A^\dagger)^\dagger = (U_aA^\dagger U_a^\dagger)^\dagger
  =U_aAU_a^\dagger = \sigma_a(A)
\end{align*}
and, consequently, that $X_a^\ast = X_a$. As noted previously, $X_a$
is a $(\sigma_a,\idA)$-derivation as well as a
$(\idA,\sigma_a)$-derivation, and we shall use this fact to construct
a $\st$-$\ast$-algebra. Let us introduce
\begin{align*}
  \Xt_k =
  \begin{cases}
    X_k&\text{ if }1\leq k\leq n\\
    X_{k-n}&\text{ if }n<k\leq 2n.
  \end{cases}
\end{align*}
such that $\Xt_k$ is a $(\sigmat_k,\taut_k)$-derivation with
\begin{align*}
  \sigmat_k=
  \begin{cases}
    \sigma_k&\text{ if }1\leq k\leq n\\
    \idA&\text{ if }n<k\leq 2n
  \end{cases}\qqand
  \taut_k=
  \begin{cases}
    \idA&\text{ if }1\leq k\leq n\\
    \sigma_{k-n}&\text{ if }n<k\leq 2n    
  \end{cases};
\end{align*}
for convenience we shall also write
\begin{align*}
  U_k =
  \begin{cases}
    U_k &\text{ if }1\leq k\leq n\\
    U_{k-n} &\text{ if }n<k\leq 2n
  \end{cases}
\end{align*}
such that $\Xt_k(A)=A-U_kAU_k^\dagger$ for $k=1,\ldots,2n$.  Defining
the involution $\iota$ as
\begin{align*}
  \iota(k) =
  \begin{cases}
    k+n&\text{ if }1\leq k\leq n\\
    k-n&\text{ if }n<k\leq 2n
  \end{cases}
\end{align*}
it is easy to verify that 
\begin{align*}
  \Sigmas = (\MatNC,\{\Xt_i\}_{i=1}^{2n},\iota)
\end{align*}
is a $\st$-$\ast$-algebra; i.e.
\begin{align*}
  &\sigmat_{\iota(k)}^\ast =
  \begin{cases}
    \sigmat_{k+n}=\idA &\text{ if }1\leq k\leq n\\
    \sigmat_{k-n}=\sigma_{k-n}&\text{ if }n<i\leq 2n    
  \end{cases}
  =\taut_k\\
  &\Xt_{\iota(k)} =
    \begin{cases}
      \Xt_{k+n}&\text{ if }1\leq k\leq n\\
      \Xt_{k-n}&\text{ if }n<k\leq 2n
    \end{cases}=
    \begin{cases}
    X_k=X_k^\ast &\text{ if }1\leq k\leq n\\
    X_{k-n}=X_{k-n}^\ast&\text{ if }n<k\leq 2n
  \end{cases}
  =\Xt_k^\ast.
\end{align*}
Note that even though formally the number of $\st$-derivations has
been doubled, the tangent space remains the same,
i.e. $\TSigma=\TSigma^\ast$, and $\TSigma^\ast$ can be given the same
structure of a $\st$-Lie algebra as $\TSigma$.

One can construct a free $\Sigma^\ast$-module
$(\MatNC,\{\sigmah_k,\tauh_k\}_{k=1}^{2n})$ by setting
\begin{equation*}
  \sigmah_k(A)=\sigmat_k(A)\qand
  \tauh_k(A)=\taut_k(A)
\end{equation*}
for $k=1,\ldots,2n$. A (left) $\st$-connection on
$(\MatNC,\{\sigmah_k,\tauh_k\}_{k=1}^{2n})$ has to satisfy
\begin{align*}
  \nabla_{\Xt_k}(BA) = \sigmat_k(B)\nabla_{\Xt_k}A + \Xt_k(B)\tauh_k(A)
\end{align*}
for $k=1,\ldots,2n$, which is equivalent to
\begin{align}
  &\nabla_{X_a}(BA) = \sigma_a(B)\nabla_{X_a}A + X_a(B)A\label{eq:nabla.left.BA}\\
  &\nabla_{X_a}(BA) = B\nabla_{X_a}A + X_a(B)\sigma_a(A)\label{eq:nabla.right.BA}
\end{align}
for $a=1,\ldots,n$. Hence, doubling the $\st$-derivations has the
effect that one requires the connection to satisfy two different product rules
in analogy with $X_a$ being both a $(\sigma_a,\idA)$-derivation and a
$(\idA,\sigma_a)$-derivation. Let us now show that under a mild
regularity assumption, there is a unique connection on the free
$\Sigma^\ast$-module $(\MatNC,\{\sigmah_a,\tauh_a\}_{a=1}^{2n})$.
\begin{definition}
  A $\st$-derivation $X$ on $\MatNC$ is called regular if there exists
  $A\in\MatNC$ such that $\det(X(A))\neq 0$. The $\st$-algebra
  $(\MatNC,\{\Xt_k\}_{k=1}^{2n})$ is called regular if $X_a$ is regular for
  $a=1,\ldots,n$.
\end{definition}

\begin{proposition}\label{prop:matrix.st.star.unique.connection}
  Assume that $(\MatNC,\{\Xt_k\}_{k=1}^{2n})$ is a regular
  $\st$-$\ast$-algebra. Then there exists a unique $\st$-connection on the free $\Sigma^\ast$-module
  $(\MatNC,\{\sigmah_a,\tauh_a\}_{a=1}^{2n})$, given by
  \begin{equation}\label{eq:matrix.unique.connection}
    \nabla_{\Xt_k}A = \Xt_k(A) = A - U_kAU_k^\dagger
  \end{equation}
  for $A\in\MatNC$ and $k=1,\ldots,2n$.
\end{proposition}

\begin{proof}
  In principle, a $\st$-connection is given by specifying
  $\nabla_{\Xt_k}\mid=\Gammat_k$ for $k=1,\ldots,2n$. However, since
  $\Xt_{a}=\Xt_{a+n}$ for $a=1,\ldots,n$, it is enough to specify
  $\nabla_{X_a}\mid=\Gammat_a$ for $a=1,\ldots,n$.  As previously
  noted, the requirement that
  \begin{align*}
    \nabla_{\Xt_k}(BA) = \sigmat_k(B)\nabla_{\Xt_k}A + \Xt_k(B)\tauh_k(A)
  \end{align*}
  for $k=a$ and $k=a+n$ implies that
  \begin{align*}
    &\nabla_{X_a}(BA) = \sigma_a(B)\nabla_{X_a}A + X_a(B)A\\
    &\nabla_{X_a}(BA) = B\nabla_{X_a}A + X_a(B)\sigma_a(A)
  \end{align*}
  for $A,B\in\MatNC$ and $a=1,\ldots,n$. Writing
  $\nabla_{X_a}\mid = \Gammat_a$ and equating the above expressions
  for $A=\mid$, gives
  \begin{align*}
    0 &= B\nabla_{X_a}\mid + X_a(B)-\sigma_a(B)\nabla_{X_a}\mid -X_a(B)\\        
    &= (B-\sigma_a(B))\Gammat_a = X_a(B)\Gammat_a,
  \end{align*}
  for all $B\in\MatNC$ and $a=1,\ldots,n$.  Since $\Sigma^\ast$ is
  assumed to be regular, there exist $B_a\in\MatNC$ for
  $a=1,\ldots,n$, such that $\det(X_a(B_a))\neq 0$, implying that
  $\Gammat_a=0$ for $a=1,\ldots,n$, giving $\Gammat_k=0$ for
  $k=1,\ldots,2n$ and the $\st$-connection
  \begin{align*}
    \nabla_{\Xt_k}(A) = \Xt_k(A)
  \end{align*}
  for $k=1,\ldots,2n$.
\end{proof}

\noindent
The unique connection in
Proposition~\ref{prop:matrix.st.star.unique.connection} is torsion
free with respect to the anchor map $\varphi(X_a)=E_a$ with
$[E_a,E_b] = [E_a,U_b]=0$, since
\begin{align*}
  \nabla_{X_a}\varphi(X_b) = E_b-U_aE_bU_a^\dagger = E_b-E_bU_aU_a^\dagger = 0
\end{align*}
for $a,b=1,\ldots,n$, and consequently
\begin{align*}
  \nabla_{\Xt_k}\varphi(\Xt_l)-\nabla_{\Xt_l}\varphi(\Xt_k) = 0
\end{align*}
for $k,l=1,\ldots,2n$. Moreover, the connection $p\circ\nabla$ is a
torsion free $\st$-connection on
$(\MatNC p,\{p\circ\sigmah_k,p\circ\tauh_k\}_{k=1}^{2n})$ with respect
to $\varphi(X_a)=E_ap$ if $p=v_0v_0^\dagger$, and $v_0$ is chosen to
be a common eigenvector of $\{E_a,U_a\}_{a=1}^n$; namely
\begin{align*}
  \nabla_{X_a}\varphi(X_b)
  &= E_bp - U_aE_bpU_a^\dagger p= E_b(p-U_apU_a^\dagger p) \\
  &= E_b(p-\mu_a p\mu_a^{-1} p) = E_b(p-p^2) = 0
\end{align*}
for $a,b=1,\ldots,n$.

\subsection{Levi-Civita $\st$-connections}

\noindent
A hermitian form on the (left) module $\MatNC$ is given as
\begin{align}\label{eq:metric.matrix.free}
  h(A,B) = Ah_0B^\dagger
\end{align}
for $h_0\in\MatNC$ such that $h_0^\dagger=h_0$. Demanding that $h$ is
$\st$-invariant on the $\Sigma^\ast$-module
$(\MatNC,\{\sigmah_k,\tauh_k\}_{k=1}^{2n})$ amounts to
\begin{align*}
  \sigma_a(h_0) = h_0\equivalent [U_a,h_0]=0,
\end{align*}
implying that $X_a(h_0)=h_0-\sigma_a(h_0)=0$ for $a=1,\ldots,n$. For
instance, for $1\leq k\leq n$ one checks that
\begin{align*}
  &\sigmat_k\paraa{h(A,B)} = h\paraa{\sigmah_k(A),\tauh_{\iota(k)}(B)}\equivalent
  \sigma_k(Ah_0B^\dagger) = \sigma_k(A)h_0\sigma_k(B)\equivalent\\
  &\sigma_k(A)(h_0-\sigma_a(h_0))\sigma_k(B) = 0.
\end{align*}
The unique connection in
Proposition~\ref{prop:matrix.st.star.unique.connection} is compatible
with the hermitian form $h$; namely, for $1\leq k\leq n$ one finds that
\begin{align*}
  h\paraa{\sigmah_k(A)&,\nabla_{\Xt_{\iota(k)}}B}
    +h\paraa{\nabla_{\Xt_k}A,\sigmah_{\iota(k)}(B)}
  = \sigma_k(A)h_0X_k(B)^\dagger + X_k(A)h_0 B^\dagger\\
                      &= \sigma_k(A)h_0X_k(B^\dagger) + X_k(A)h_0 B^\dagger
                        =\sigma_k(A)\sigma(h_0)X_k(B^\dagger) + X_k(A)h_0 B^\dagger\\
                        &= \sigma_k(A)X_k(h_0B^\dagger) + X_k(A)h_0 B^\dagger
                      = X_k\paraa{h(A,B)}
\end{align*}
using that $X_k^\ast=X_k$, $X_k(h_0)=0$ and $\sigma_k(h_0)=h_0$. One
can do a similar computation for $n<k\leq 2n$, showing that the
$\st$-connection is indeed compatible with the hermitian form
$h$. Hence, $\nabla$ is a Levi-Civita $\st$-connection on
$(\MatNC,\{\sigmah_k,\tauh_k\}_{k=1}^{2n})$.

Next, let us construct a Levi-Civita $\st$-connection on the
projective module
\begin{equation*}
 (\MatNC p,\{p\circ\sigmah_k,p\circ\tauh_k\}_{k=1}^{2n}). 
\end{equation*}
As already noted, the connection
\begin{align*}
  p\paraa{\nabla_{\Xt_k}A} = \Xt_k(A)p
\end{align*}
is a torsion free connection with respect to $\varphi(X_a)=E_ap$.
Proposition~\ref{prop:metric.connection.projective.module} implies
that $p\circ\nabla$ is a metric connection on
$(\MatNC p,\{p\circ\sigmah_k,p\circ\tauh_k\}_{k=1}^{2n})$ if $p$ is an
orthogonal projection with respect to $h$, and
$[\sigmah_k,p]=[\tauh_k,p]=0$. If $v_0$ is an eigenvector of $U_a$,
with eigenvalue $\mu_a$, then
\begin{align*}
  [U_a,p] &= Uv_0v_0^\dagger-v_0v_0^\dagger U_a
            = \mu_ap - v_0(U_a^\dagger v_0)
            = \mu_ap - v_0(\bar{\mu}_av_0)^\dagger\\
          &= \mu_ap - \mu_ap = 0,
\end{align*}
which implies that $[\sigmah_k,p]=[\tauh_k,p]=0$. In order for $p$ to
be an orthogonal projection with respect to $h$, one needs to choose
$h_0$ such that $[h_0,p]=0$ since 
\begin{align*}
  h(Ap,B)-h(A,Bp)
  &= Aph_0B^\dagger-Ah_0(Bp^\dagger)
    = Aph_0B^\dagger - Ah_0pB^\dagger\\
  &=A(ph_0-h_0p)B^\dagger.
\end{align*}
For instance, one may choose $h=\hat{h}_0p$ for arbitrary $\hat{h}_0\in\reals$.
Written in terms of vectors in $\complex^N$, one finds that
\begin{align*}
  \nabla_{X_a}v = (\mid-\bar{\mu}_aU_a)v
\end{align*}
is a Levi-Civita $\st$-connection with respect to
$\varphi(X_a)=\lambda_av_0$ and
\begin{align*}
  h(u,v)=\hat{h}_0uv^\dagger
\end{align*}
for arbitrary $\hat{h}_0\in\complex$ and $\lambda_a\in\complex$, where
$v_0$ is a common eigenvector of $\{U_a\}_{a=1}^n$ with $U_av_0=\mu_av_0$.

\section*{Acknowledgements}

\noindent
J.A. is supported by the Swedish Research Council grant no. 2017-03710.

\bibliographystyle{alpha}
\bibliography{references}

\end{document}